\documentclass[11pt, reqno]{amsart}
\usepackage{amsmath, cite, url, amsthm, amssymb, amsfonts, mathrsfs, times}  
\usepackage{latexsym}
\usepackage{mathdots}
\usepackage{amssymb}
\usepackage{graphicx,bm,units,yfonts}
\usepackage[usenames,dvipsnames]{xcolor}
\usepackage{hyperref}

\newtheorem{theorem}[equation]{Theorem}
\newtheorem{corollary}[equation]{Corollary}
\newtheorem{lemma}[equation]{Lemma}
\newtheorem{proposition}[equation]{Proposition}

\theoremstyle{definition}
\numberwithin{equation}{section}
\newtheorem{example}[equation]{Example}
\newtheorem{definition}[equation]{Definition}

\newcommand{\itemizespacing}{\addtolength{\itemsep}{0.5\baselineskip}}

\newcommand{\codim}{\operatorname{codim}}
\newcommand{\C}{\mathbb{C}}
\newcommand{\dom}{\mathcal{D}}
\newcommand{\h}{\mathcal{H}}

\newcommand{\E}{\mathcal{E}}
\newcommand{\K}{\mathcal{K}}

\newcommand{\N}{\mathbb{N}}
\newcommand{\R}{\mathbb{R}}
\newcommand{\T}{\mathbb{T}}
\newcommand{\V}{\mathcal{V}}
\newcommand{\Z}{\mathbb{Z}}

\newcommand{\diag}{\operatorname{diag}}
\newcommand{\ran}{\operatorname{ran}}
\newcommand{\inner}[1]{\langle #1 \rangle}
\newcommand{\norm}[1]{\left\| #1 \right\|}
\newcommand{\minimatrix}[4]{\begin{bmatrix} #1 & #2 \\ #3 & #4 \end{bmatrix}  }
\newcommand{\megamatrix}[9]{\begin{bmatrix} #1 & #2 & #3 \\ #4 & #5  & #6 \\ #7 & #8 & #9 \end{bmatrix}  }
\newcommand{\threevector}[3]{ \begin{bmatrix} #1 \\ #2 \\ #3 \end{bmatrix} }
\newcommand{\blf}{ {[\,\cdot\, , \,\cdot\,]} }
\newcommand{\sesqui}{ {\langle\,\cdot\, , \,\cdot\,\rangle} }

\renewcommand{\Re}{\operatorname{Re}}
\renewcommand{\Im}{\operatorname{Im}}
\renewcommand{\vec}[1]{{\bf #1}}

\begin{document}

\title{Mathematical and physical aspects of complex symmetric operators}

\author{Stephan Ramon Garcia}
\address{Department of Mathematics, Pomona College, Claremont, CA, 91711, USA.}
\email{stephan.garcia@pomona.edu}
\urladdr{http://pages.pomona.edu/~sg064747/}

\author{Emil Prodan}
\address{Department of Physics, Yeshiva University, New York, NY 10016, USA}
\email{prodan@yu.edu}

\author{Mihai Putinar}
\address{Department of Mathematics, University of California at Santa Barbara, CA 93106, USA, {\it and}
Division of Mathematical Sciences, Nanyang Technological University, Singapore 637371}
\email{mputinar@math.ucsb.edu} 
\email{mputinar@ntu.edu.sg}

\begin{abstract} 
Recent advances in the theory of complex symmetric operators are presented and related to current studies in non-hermitian quantum mechanics. The main themes of the survey are:  the structure of complex symmetric operators, $C$-selfadjoint extensions of $C$-symmetric unbounded operators, resolvent estimates, reality of spectrum, bases of $C$-orthonormal vectors, and conjugate-linear symmetric operators. The main results are complemented by a variety of natural examples arising in field theory, quantum physics, and complex variables.
\end{abstract}

\maketitle
%\clearpage
%\tableofcontents
%\clearpage
\section{Introduction}

The study of complex symmetric operators has been flourished near the intersection of
	operator theory and complex analysis.
	The general study of complex symmetric operators was undertaken by the first author, third author, and W.R.~Wogen (in various combinations)
	in \cite{CSOA, CSO2, ESCSO, MUCFO, ICS, CSPI, SNCSO}.
	A number of other authors
	have recently made significant contributions to the study of complex symmetric operators
	\cite{CFT, Gilbreath, ZLJ, ZL, JKLL, JKLL2,WXH,Tener,Zag,ZLZ}, which has proven
	particularly relevant to the study of truncated Toeplitz operators \cite{CRW, Chalendar, RPTTO, NLEPHS, 
	SedlockThesis, Sedlock, TTOSIUES }, a rapidly growing branch of function-theoretic operator theory 
	stemming from the seminal work of D.~Sarason \cite{Sarason}.

The last decade witnessed a revived interest in non-hermitian quantum mechanics and in the spectral analysis of 
	certain complex symmetric operators. The proliferation of publications and scientific meetings devoted to the subject leaves the mathematicians 
	and the mathematical aspects of the theory far behind. As incomplete and subjective as it may be, our survey aims 
	at connecting the communities of mathematicians and physicists on their common interest in complex symmetric operators.
	Having in mind a non-expert reader with inclination towards mathematical physics, we proceed at a non-technical level, indicating instead 
	precise bibliographical sources.  Among the recently published monographs dealing at least tangentially 
with complex symmetry we mention \cite{Moiseyev} devoted to resonance theory arising in quantum mechanics and the thesis \cite{NesemannPT} where a clear link between complex symmetric operators and spaces with an indefinite metric is unveiled.  The reader may also wish to consult the recent special issue on non-Hermitian
	quantum physics published in the \emph{Journal of Physics A: Mathematical and Theoretical} \cite{Special}.
%	(vol.~{\bf 45}, no.~44, 2012).

%If you would like to reference the recent special issue on quantum physics with non-Hermitian operators (e.g. in the first paragraph of the introduction) rather than any specific papers in it, please use the reference of the preface: Carl Bender et al 2012 J. Phys. A: Math. Theor. 45 440301 doi:10.1088/1751-8113/45/44/440301
%Thank you.

	The study of  \emph{complex symmetric} (i.e., self-transpose) matrices has deep classical roots, 
	stretching back to the work of L.-K.~Hua on automorphic functions \cite{Hua}, N.~Jacobson on projective 
	geometry \cite{Jacobson}, I.~Schur on quadratic forms \cite{Schur}, C.L.~Siegel on symplectic geometry 
	\cite{Siegel}, and T.~Takagi on function theory \cite{Takagi}.  The connection between 
	complex symmetric matrices and the study of univalent functions emerged in the early 1980s
	\cite{DurenUnivalent, Fitzgerald, HJ}.
	Nevertheless, complex symmetric matrices as a whole have not received the attention which they deserve.
	The modern text \cite[Ch.~4.4]{HJ} and the classic \cite[Ch.~XI]{Gantmacher} are among the few places
	where complex symmetric matrices are discussed in the textbook literature.

         The pioneering work of Glazman \cite{Glazman,GlazmanBook}
marks the foundation of the extension
theory of complex symmetric differential operators; see also \cite{Zihar,Visik}. 
Glazman's work was complemented in a series of articles \cite{Knowles,Race,McLeod} offering a detailed analysis of the boundary conditions for Sturm-Liouville operators that enjoy complex symmetry.
The parallel to the theory of symmetric operators
in an indefinite metric space is natural and necessary; both symmetries have the form $T \subseteq ST^\ast S$, with a
conjugate-linear involution in the first case, and a unitary involution in the second. Later on, complex symmetric
operators and symmetric operators with respect to an indefinite metric merged into a powerful modern construct
\cite{AK1, AK2, AK3, AGK,LangerTretter}.

	In the realm of applied mathematics, complex symmetric matrices appear in the study of 
	quantum reaction dynamics \cite{Bar-On, Brandas}, electric power modeling \cite{Howle},
	the numerical simulation of high-voltage insulators \cite{Reitzinger}, 
	magnetized multicomponent transport \cite{Giovangigli},
	thermoelastic wave propagation \cite{Scott}, the maximum clique problem in 
	graph theory \cite{Budinich}, elliptically polarized plane waves in continuous media \cite{MR1284711},
	inverse spectral problems for semisimple damped vibrating systems \cite{Lancaster}, 
	low-dimensional symplectic gravity models \cite{MR1642302},
	the study of decay phenomena \cite{MR1022502}, scattering matrices in atomic collision theory \cite{Brown},
	and the numerical solution of the time-harmonic Maxwell equation in axisymmetric cavity surface 
	emitting lasers \cite{Arbenz}.  Throughout the years, complex symmetric matrices 
	have also been the focus of sporadic numerical work \cite{Arbenz2, MR1887661, MR1663498, MR1474652, 
	MR1145195, MR1243819, MR1766915, MR1262780, MR1342946, MR2098041, MR1483849, 
	MR1923717, MR1203022}.

We aim here to discuss the general mathematical properties of complex symmetric operators, 
keeping an eye on those aspects of the theory that may be more appealing to the mathematical physicist.
Proofs are given when convenient, although much of the time we will simply provide the reader with a
sketch of the proof or a reference.

\subsection*{Disclaimer}
	Given the widespread recent interest in non-selfadjoint operators from the mathematical physics community, 
	it is likely that some of the results presented here already exist in the physics literature.   A rapid count on the American Mathematical Society scientific net (MathSciNet) gives more than 200 articles solely devoted to $\mathcal{PT}$-symmetric operators.
	We are simply trying to help bridge the gap between the growing community of mathematical physicists working
	on non-selfadjoint operators with our own community of operator theorists who study complex symmetric operators for their own sake.
	If we have omitted any key references or major results, then we apologize.
	
	We must also confess that in writing this survey article, we have borrowed freely from our own previously published work.
	In particular, we have engaged in vigorous recycling of material from our articles
	\cite{CSOA, CSO2, CCO, ESCSO, AAEPRI, VPSBF, CSQ}, although we have taken great care to
	streamline our presentation and standardize the notation used throughout this article.

\subsection*{Notation}
	We adopt the customary notation used in the mathematics literature.
	For instance, our inner products are linear in the first slot and we
	use $\overline{z}$ instead of $z^*$ to denote complex conjugation.	
	Vectors in an abstract Hilbert space will be most often written in bold (e.g., $\vec{v}$) as opposed
	to italic (e.g., $v$).  On the other hand, vectors in concrete Hilbert spaces, such as $L^2(\R)$,
	will be denoted as appropriate for that setting.	

	Matrices and operators shall be denote by upper-case letters such as $A,B,\ldots$
	and scalars by lower-case letters $a,b,\ldots$ or their Greek equivalents $\alpha,\beta,\ldots$.
	We let $I$ denote the identity operator and we 	
	use $A^*$ instead of $A^{\dagger}$ to denote the adjoint of $A$.  The superscript $T$, as in $A^T$, will denote
	the transpose of a matrix.	
	
	We say that two operators $A$ and $B$ are said to be \emph{unitarily equivalent} if there exists a unitary
	operator $U$ such that $A = UBU^*$.  We denote this by $A \cong B$, noting that $\cong$ is an equivalence
	relation (in the matrix-theory literature, the term \emph{unitarily similar} is preferred).
	The norm $\norm{A}$ of an operator always refers to the operator norm $\norm{A} = \sup_{\norm{\vec{x}}=1} \norm{A\vec{x}}$.

\subsection*{Acknowledgments}
S.R.~Garcia acknowledges the support of NSF Grants DMS-1001614 and DMS-1265973. E.~Prodan was supported by 
NSF grants DMS-1066045 and DMR-1056168. M.~Putinar was partially supported by a Grant from Nanyang Technological University. 
We are indebted to David Krej\v{c}i\v{r}\'{i}k and Miroslav Znojil for constructive criticism and precious bibliographical guidance.

\section{Complex symmetric operators}
	Since complex symmetric operators are characterized by their interactions with certain conjugate-linear operators,
	we begin with a brief discussion of these auxiliary operators.

\subsection{Conjugations}
	The following concept is a straightforward generalization of complex conjugation $z \mapsto \overline{z}$, which itself
	can be viewed as a conjugate-linear map on the one-dimensional Hilbert space $\C$. 

	\begin{definition}
		A \emph{conjugation} on a complex Hilbert space $\h$ is a function $C:\h\to\h$ that is
		\begin{enumerate}\itemizespacing
			\item conjugate-linear: $C(\alpha \vec{x} + \beta \vec{y}) = \overline{\alpha} C\vec{x} + \overline{\beta}C\vec{y}$
				for all $\vec{x},\vec{y}$ in $\h$,
			\item involutive: $C^2 = I$,
			\item isometric: $\norm{C\vec{x}} = \norm{\vec{x}}$ for all $\vec{x}$ in $\h$.
		\end{enumerate}
	\end{definition}
	
	The relevance of conjugations to the extension theory for unbounded symmetric (i.e., $T \subseteq T^*$) operators
	was recognized by von Neumann, who realized that a symmetric densely defined
	operator $T:\dom(T)\to\h$ that is \emph{$C$-real} (i.e., $T = CTC$) admits
	selfadjoint extensions \cite{vonNeumann}.  In the theory of von Neumann algebras, 
	conjugations feature prominently in the Tomita-Takesaki modular theory for Type III factors and thus
	in the noncommutative geometry program initiated by A.~Connes \cite{ConnesBook}.
		
	Some authors prefer to use the term \emph{antilinear} instead of \emph{conjugate-linear}.
	From this perspective, a function that satisfies the first and third conditions listed above is called an 
	\emph{antiunitary} operator.  A conjugation is simply an antiunitary operator
	that is involutive.  In light of the polarization identity
	\begin{equation*}
		4 \inner{ \vec{x}, \vec{y} } = \norm{\vec{x} + \vec{y}}^2  -\norm{\vec{x} - \vec{y}}^2  
		+i\norm{\vec{x} + i\vec{y}}^2  -i\norm{\vec{x} -i \vec{y}}^2,
	\end{equation*}
	the isometric condition is equivalent to asserting that $\inner{C\vec{x},C\vec{y}} = \inner{\vec{y},\vec{x}}$ 
	for all $\vec{x},\vec{y}$ in $\h$.  
	Let us consider a few standard examples of conjugations.
	
	\begin{example}\label{Example:MS}
		If $(X,\mu)$ is a measure space (with $\mu$ a positive measure on $X$), then the \emph{canonical conjugation} on $L^2(X,\mu)$ is just
		pointwise complex conjugation: 
		\begin{equation*}
			[Cf](x) = \overline{f(x)}.
		\end{equation*}
		Particular instances include the canonical conjugations 
		\begin{equation}\label{eq-CanonicalConjugationFinite}
			C(z_1,z_2,\ldots,z_n) = (\overline{z_1}, \overline{z_2}, \ldots, \overline{z_n} )
		\end{equation}
		on $\C^n = \ell^2(\{1,2,\ldots,n\})$ and
		\begin{equation}\label{eq-CanonicalConjugationInfinite}
			C(z_1,z_2,z_3,\ldots) = (\overline{z_1}, \overline{z_2}, \overline{z_3}, \ldots )
		\end{equation}
		on the space $\ell^2(\N)$ of all square-summable complex sequences.
	\end{example}

	\begin{example}\label{Example:ToeplitzConjugation}
		The \emph{Toeplitz conjugation} on $\C^n$ is defined by
		\begin{equation}\label{eq-ToeplitzConjugation}
				C(z_1,z_2,\ldots,z_n) = (\overline{z_n}, \overline{z_{n-1}}, \ldots, \overline{z_1} ).
		\end{equation}	
		As its name suggests, the Toeplitz conjugation is related to the study of Toeplitz matrices.
		In light of its appearance in the \emph{Szeg\H{o} recurrence} from the theory of orthogonal
		polynomials on the unit circle (OPUC) \cite[eq.~1.1.7]{SimonOPUC}, one might also refer to \eqref{eq-ToeplitzConjugation}
		as the \emph{Szeg\H{o} conjugation}.
	\end{example}
	
	\begin{example}\label{Example:VC}
		Building upon Example \ref{Example:MS},
		if one has a measure space $(X,\mu)$ that possesses a certain amount of symmetry,
		one can sometimes form a conjugation that respects this symmetry.  For instance, the conjugation
		\begin{equation}\label{eq:ConjugationVolterra}
			[Cf](x) = \overline{ f(1-x)}
		\end{equation}
		on $L^2[0,1]$ arises in the study of certain highly non-normal integral operators
		(see Example \ref{Example:Volterra}).
	\end{example}

	\begin{example}
		Consider the \emph{parity} operator $$[\mathcal{P}\psi](\vec{x}) = \psi(-\vec{x})$$ and 
		the \emph{time-reversal} operator $$[\mathcal{T} \psi](\vec{x}) = \overline{\psi(\vec{x})}$$ on $L^2(\R^n)$.
		Since $\mathcal{T}$ is a conjugation on $L^2(\R^n)$ that commutes with $\mathcal{P}$,
		it is not hard to show that their composition $\mathcal{PT}$ is also a conjugation.  As the notation suggests,
		the conjugation $\mathcal{PT}$ plays a central role in the development of $\mathcal{PT}$-symmetric
		quantum theory \cite{MR1686605, MR1627442}. 
	\end{example}

	\begin{example}
		If the spin-degrees of freedom are considered, then we consider the Hilbert space
		$L^2(\R^n,\C^{2s+1}) \cong L^2(\R^n)\otimes \C^{2s+1}$, where $s$ is the spin of the particle. 
		The time-reversal operator now takes the form 
		$$[\mathcal{T}{\bm \psi}](\vec{x})=e^{-i\pi (1\otimes S_y)} \overline{{\bm \psi}(\vec{x})},$$ where $S_y$ is the $y$-component of the spin-operator acting on $\mathbb C^{2s+1}$. 
		For particles with integer spin number $s$ (bosons), $\mathcal{T}$ remains a conjugation. 
		Unfortunately this is not the case for particles with half-integer spin number $s$ (fermions), in which case the time-reversal operator squares to $-I$. 
		A conjugate-linear operator of this sort is called an \emph{anti-conjugation} \cite[Def.~4.1]{UET}.
	\end{example}
			
	It turns out that conjugations are, by themselves, of minimal interest.  Indeed, the following
	lemma asserts that every conjugation is unitarily equivalent to the canonical conjugation 
	on an $\ell^2$-space of the appropriate dimension.  

	\begin{lemma}\label{LemmaAllSame}
		If $C$ is a conjugation on $\h$, then there exists an orthonormal basis $\{\vec{e}_n\}$ of $\h$
		such that $C\vec{e}_n = \vec{e}_n$ for all $n$.  In particular, $C( \sum_n \alpha_n \vec{e}_n) = \sum_n \overline{\alpha_n} \vec{e}_n$
		for all square summable sequences $\{ \alpha_n\}$.
	\end{lemma}

	\begin{proof}
		Consider the $\R$-linear subspace $\K = (I + C)\h$ of $\h$ and note that
		each vector in $\K$ is fixed by $C$.  Consequently $\K$ is a \emph{real} Hilbert space under the inner product
		$\inner{\vec{x}, \vec{y} }$ since
		$\inner{ \vec{x},\vec{y}} = \inner{C\vec{y},C\vec{x}} = \inner{\vec{y},\vec{x}}  = \overline{ \inner{\vec{x},\vec{y}} }$
		for every $\vec{x},\vec{y}$ in $\K$.  Let $\{\vec{e}_n\}$ be an orthonormal basis for $\K$. 
		Since $\h = \K + i \K$, it follows easily that $\{\vec{e}_n\}$ is an orthonormal basis for the complex Hilbert space $\h$ as well.
	\end{proof}

	\begin{definition}
		A vector $\vec{x}$ that satisfies $C\vec{x} = \vec{x}$ is called a \emph{$C$-real} vector.
		We refer to a basis having the properties described in Lemma \ref{LemmaAllSame}
		as a \emph{$C$-real orthonormal basis}.
	\end{definition}
	
	\begin{example}\label{Example:CR3}
		Let $C(z_1,z_2,z_3) = (\overline{z_3}, \overline{z_2}, \overline{z_1})$ denote the Toeplitz
		conjugation \eqref{Example:ToeplitzConjugation} on $\C^3$.  Then
		\begin{equation*}
			\vec{e}_1 = (\tfrac{1}{2}, -\tfrac{1}{\sqrt{2}}, \tfrac{1}{2}),\quad
			\vec{e}_2 = (\tfrac{1}{2}, \tfrac{1}{\sqrt{2}}, \tfrac{1}{2}),\quad
			\vec{e}_3 = ( - \tfrac{i}{\sqrt{2}}, 0, \tfrac{i}{\sqrt{2}} )
		\end{equation*}
 		is a $C$-real orthonormal basis of $\C^3$.
	\end{example}
	
	\begin{example}\label{Example:VONB}
		Let $[Cf](x) = \overline{f(1-x)}$ denote the conjugation \eqref{eq:ConjugationVolterra}
		on $L^2[0,1]$.  For each $\alpha \in [0,2\pi)$, one can show that
		$$\qquad\qquad\vec{e}_n(x) = \exp[i(\alpha+2\pi n)(x-\tfrac{1}{2})],\qquad n \in \Z$$
		is a $C$-real orthonormal basis for $L^2[0,1]$
		\cite[Lem.~4.3]{CCO}.
	\end{example}

\subsection{Complex symmetric operators}

	Our primary interest in conjugations lies not with conjugations themselves, but rather
	with certain \emph{linear} operators that interact with them.  We first restrict ourselves
	to the consideration of bounded operators.  An in-depth discussion of the corresponding 
	developments for unbounded operators is carried out in Section \ref{SectionUnbounded}.

	\begin{definition}\label{DefinitionCSO}
		Let $C$ be a conjugation on $\h$.
		A bounded linear operator $T$ on $\h$ is called \emph{$C$-symmetric} if $T = CT^*C$.
		We say that $T$ is a \emph{complex symmetric operator} if there exists a $C$ with respect to 
		which $T$ is $C$-symmetric.
	\end{definition}
	
	Although the terminology introduced in Definition \ref{DefinitionCSO} is at odds
	with certain portions of the differential equations literature,
	the equivalences of the following lemma indicates that the
	term \emph{complex symmetric} is quite appropriate from
	a linear algebraic viewpoint.	
	
	\begin{lemma}\label{LemmaUECSM}
		For a bounded linear operator $T:\h\to\h$, the following are equivalent:
		\begin{enumerate}\addtolength{\itemsep}{0.5\baselineskip}
			\item $T$ is a complex symmetric operator,
			\item There is an orthonormal basis of $\h$ with respect to which $T$ has a complex symmetric
				(i.e., self-transpose) matrix representation,
			\item $T$ is unitarily equivalent to a complex symmetric matrix, acting on an $\ell^2$-space of the
				appropriate dimension.
		\end{enumerate}
	\end{lemma}

	\begin{proof}
		The equivalence of $(2)$ and $(3)$ is clear, so we focus on $(1) \Leftrightarrow (2)$.
		Suppose that $T= CT^*C$ for some conjugation $C$ on $\h$ and let $\vec{e}_n$ be a $C$-real
		orthonormal basis for $\h$ (see Lemma \ref{LemmaAllSame}).  Computing the matrix entries $[T]_{ij}$
		of $T$ with respect to $\{ \vec{e}_n\}$ we find that
		\begin{equation*}
			[T]_{ij} = \inner{T \vec{e}_j, \vec{e}_i} = \inner{CT^*C \vec{e}_j,\vec{e}_i} =
			\inner{C \vec{e}_i, T^* C\vec{e}_j} = $$ $$\inner{ \vec{e}_i, T^* \vec{e}_j} = \inner{ T \vec{e}_i, \vec{e}_j} =
			 [T]_{ji},
		\end{equation*}
		which shows that $(1) \Rightarrow (2)$.  A similar computation shows that if $\{ \vec{e}_n\}$
		is an orthonormal basis of $\h$ with respect to which $T$ has a complex symmetric matrix
		representation, then the conjugation $C$ which satisfies $C\vec{e}_n = \vec{e}_n$ for all $n$ also 
		satisfies $T = CT^*C$.
	\end{proof}
	
	We refer to a square complex matrix $A$ that equals its own transpose $A^T$
	as a \emph{complex symmetric matrix}.  As Lemma \ref{LemmaUECSM} indicates,
	a bounded linear operator is \emph{complex symmetric}, in the sense of Definition \ref{DefinitionCSO},
	if and only if it can be represented as a complex symmetric matrix with respect to some
	orthonormal basis of the underlying Hilbert space.  Thus there is a certain amount of agreement between
	the terminology employed in the matrix theory and in the Hilbert space contexts.
	The excellent book \cite{HJ}, and to a lesser extent the classic text \cite{Gantmacher}, are among
	the few standard matrix theory texts to discuss complex symmetric matrices in any detail.
	
	\begin{example}
		A square complex matrix $T$ is called a \emph{Hankel matrix} if its entries are constant along the perpendiculars
		to the main diagonal (i.e., the matrix entry $[T]_{ij}$ depends only upon $i+j$).  Infinite Hankel matrices 
		appear in the study of moment problems, control theory, approximation theory, 
		and operator theory \cite{Peller, Nikolski1, Nikolski2}.	
		Being a complex symmetric matrix, it is clear that each Hankel matrix $T$ satisfies
		$T = CT^*C$,
		where $C$ denotes the canonical conjugation \eqref{eq-CanonicalConjugationInfinite}
		on an $\ell^2$-space of the appropriate dimension.
	\end{example}
	
	\begin{example}
		The building blocks of any bounded normal operator (i.e., $T^*T = TT^*$) are the multiplication
		operators $[M_z f](z) = zf(z)$ on $L^2(X,\mu)$ where $X$ is a compact subset of $\C$ and $\mu$ is a positive Borel measure on $X$.  Since
		$M_z  = CM_z^* C$ where $C$ denotes complex conjugation in $L^2(X,\mu)$,
		it follows that every normal operator is a complex symmetric operator.
	\end{example}

	\begin{example}
		It is possible to show that every operator on a two-dimensional Hilbert space is complex symmetric.
		More generally, every \emph{binormal} operator is a complex symmetric operator \cite{SNCSO}.
	\end{example}

	\begin{example}\label{Example:ToeplitzMatrix}
		A $n \times n$ matrix $T$ is called a \emph{Toeplitz matrix} if its entries are constant along the
		parallels to the main diagonal (i.e., the matrix entry $[T]_{ij}$ depends only upon $i-j$).
		The pseudospectra of Toeplitz matrices have been the subject of much recent work \cite{Trefethen}
		and the asymptotic behavior of Toeplitz matrices and their determinants is a beautiful and well-explored territory \cite{Bottcher}.
		Generalizations of finite Toeplitz matrices are \emph{truncated Toeplitz operators}, a subject
		of much interest in function-related operator theory \cite{Sarason, RPTTO}.
		Our interest in Toeplitz matrices stems from the fact that 
		every \emph{finite} Toeplitz matrix $T$ satisfies $T = CT^*C$ where $C$ denotes the \emph{Toeplitz conjugation}
		\eqref{eq-ToeplitzConjugation}.
	\end{example}

	\begin{example}
		The question of whether a given operator is actually a complex symmetric operator is more subtle
		than it first appears.  For instance, one can show that among the matrices
		\begin{equation}\label{eq-TenerMatrices}
			\megamatrix{0}{7}{0}{0}{1}{2}{0}{0}{6} \quad
			\megamatrix{0}{7}{0}{0}{1}{3}{0}{0}{6} \quad
			\megamatrix{0}{7}{0}{0}{1}{4}{0}{0}{6} \quad
			\megamatrix{0}{7}{0}{0}{1}{5}{0}{0}{6} \quad
			\megamatrix{0}{7}{0}{0}{1}{6}{0}{0}{6},
		\end{equation}	
		all of which are \emph{similar} to the diagonal matrix $\operatorname{diag}(0,1,6)$, only 
		the fourth matrix listed in \eqref{eq-TenerMatrices} is unitarily
		equivalent to a complex symmetric matrix \cite{Tener}.
		A particularly striking example of such an unexpected 
		unitary equivalence is 
		\begin{equation*}
			\left[
			\begin{array}{ccc}
				9 & 8 & 9 \\0 & 7 & 0 \\0 & 0 & 7
			\end{array}
			\right] \cong
			\left[\tiny
			\begin{array}{ccc}
				 8-\frac{\sqrt{149}}{2} & \frac{9}{2} i \sqrt{\frac{16837+64 \sqrt{149}}{13093}} 
				 & i \sqrt{\frac{133672}{13093}-\frac{1296 \sqrt{149}}{13093}} \\
				 \frac{9}{2} i \sqrt{\frac{16837+64 \sqrt{149}}{13093}} & \frac{207440+9477 \sqrt{149}}{26186} 
				 & \frac{18 \sqrt{3978002+82324 \sqrt{149}}}{13093} \\
				 i \sqrt{\frac{133672}{13093}-\frac{1296 \sqrt{149}}{13093}} 
				 & \frac{18 \sqrt{3978002+82324 \sqrt{149}}}{13093} & \frac{92675+1808 \sqrt{149}}{13093}
			\end{array}
			\right].
		\end{equation*}
		In particular, observe that a highly non-normal operator may possess rather 
		subtle hidden symmetries.  Algorithms to detect and exhibit such unitary equivalences have been discussed at length in 
		\cite{UECSMGC, UECSMLD, UECSMMC, Tener, MR3034497}.
	\end{example}
	
	\begin{example}\label{Example:Volterra}
		The \emph{Volterra operator} and its adjoint
		\begin{equation*}
			[Tf](x) = \int_0^x f(y)\,dy  ,\qquad  [T^*f](x) = \int_x^1 f(y)\,dy,
		\end{equation*}
		on $L^2[0,1]$ 
		satisfy $T = CT^*C$ where  $[Cf](x) = \overline{f(1-x)}$ denotes the conjugation
		from Example \ref{Example:VC}.  The orthonormal basis 
		\begin{equation*}\qquad\qquad
			\vec{e}_n = \exp\left[2 \pi  i n \left(x - \tfrac{1}{2}\right)\right], \qquad n \in \Z
		\end{equation*}
		of $L^2[0,1]$ is $C$-real (see Example \ref{Example:VONB}).
		The matrix for $T$ with respect to the basis $\{\vec{e}_n\}_{n\in \Z}$ is
		\begin{equation*}\small
			\left[
			\begin{array}{ccccccccc}
				\ddots& \vdots & \vdots & \vdots & \vdots & \vdots & \vdots & \vdots & \iddots \\[2pt]
				\cdots &   \frac{i}{6 \pi } & 0 & 0 & \frac{i}{6 \pi } & 0 & 0 & 0 &  \cdots \\[5pt]
				\cdots &   0 & \frac{i}{4 \pi } & 0 & -\frac{i}{4 \pi } & 0 & 0 & 0 &  \cdots \\[5pt]
				\cdots &   0 & 0 & \frac{i}{2 \pi } & \frac{i}{2 \pi } & 0 & 0 & 0 &  \cdots \\[5pt]
				 \cdots  & \frac{i}{6 \pi } & -\frac{i}{4 \pi } & \frac{i}{2 \pi } 
				 & \boxed{\tfrac{1}{2}}  & -\frac{i}{2 \pi } & \frac{i}{4 \pi } & -\frac{i}{6 \pi }& \cdots \\[5pt]
				\cdots &   0 & 0 & 0 & -\frac{i}{2 \pi } & -\frac{i}{2 \pi } & 0 & 0 &  \cdots \\[5pt]
				\cdots &   0 & 0 & 0 & \frac{i}{4 \pi } & 0 & -\frac{i}{4 \pi } & 0 &  \cdots \\[5pt]
				\cdots &   0 & 0 & 0 & -\frac{i}{6 \pi } & 0 & 0 & -\frac{i}{6 \pi } &  \cdots \\[2pt]
				\iddots & \vdots & \vdots & \vdots & \vdots & \vdots & \vdots & \vdots &  \ddots \\
			\end{array}
			\right],
		\end{equation*}	
		which is complex symmetric (i.e., self-transpose).
	\end{example}

	\begin{example}\label{Example:3x3T}
		Building upon Examples \ref{Example:ToeplitzConjugation} and \ref{Example:ToeplitzMatrix}, we
		see that a $3 \times 3$ nilpotent Jordan matrix $T$ satisfies $T = CT^*C$, where
		\begin{equation*}
			T = \megamatrix{0}{1}{0}{0}{0}{1}{0}{0}{0},
			\qquad
			C\threevector{z_1}{z_2}{z_3} = \threevector{ \overline{z_3} }{ \overline{z_2} }{ \overline{z_1} }.
		\end{equation*}
		Let $\{ \vec{e}_1, \vec{e}_2, \vec{e}_3\}$ denote the $C$-real orthonormal basis for $\C^3$
		obtained in Example \ref{Example:CR3}
		and form the unitary $U = [ \vec{e}_1| \vec{e}_2 | \vec{e}_3]$, yielding
		\begin{equation*}
			U = \left[
			\begin{array}{c|c|c}
			 \frac{1}{2} & \frac{1}{2} & -\frac{i}{\sqrt{2}} \\
			 -\frac{1}{\sqrt{2}} & \frac{1}{\sqrt{2}} & 0 \\
			 \frac{1}{2} & \frac{1}{2} & \frac{i}{\sqrt{2}}
			\end{array}
			\right],
			\qquad
			U^*TU = \left[
			\begin{array}{ccc}
			 -\frac{1}{\sqrt{2}} & 0 & -\frac{i}{2} \\
			 0 & \frac{1}{\sqrt{2}} & \frac{i}{2} \\
			 -\frac{i}{2} & \frac{i}{2} & 0
			\end{array}\right].
		\end{equation*}
	\end{example}

	The following folk theorem is well-known and has been rediscovered many times \cite{HJ, CSOA, Gantmacher}.

	\begin{theorem}\label{TheoremFolk}
		Every finite square matrix is similar to a complex symmetric matrix.
	\end{theorem}
	
	\begin{proof}
		Every matrix is similar to its Jordan canonical form.  A suitable generalization of Example \ref{Example:3x3T}
		shows that every Jordan block is unitarily equivalent (hence similar) to a complex symmetric matrix.
	\end{proof}

	The preceding theorem illustrates a striking contrast between the theory of selfadjoint
	matrices (i.e., $A = A^*$) and complex symmetric matrices (i.e., $A = A^T)$.  The Spectral
	Theorem asserts that every selfadjoint matrix has an orthonormal basis of eigenvectors
	and that its eigenvalues are all real.  On the other hand, a complex symmetric matrix may have any
	possible Jordan canonical form.  This extra freedom arises from the fact that it takes
	$n^2+n$ real parameters to specify a complex symmetric matrix, but only $n^2$ real parameters
	to specify a selfadjoint matrix.  The extra degrees of freedom occur due to the fact that the
	diagonal entries of a selfadjoint matrix must be real, whereas there is no such restriction upon
	the diagonal entries of a complex symmetric matrix.
	
\subsection{Bilinear forms}\label{Subsection:Bilinear}
	Associated to each conjugation $C$ on $\h$ is the bilinear form
	\begin{equation}\label{eq:BLF}
		[\vec{x},\vec{y} ] = \inner{ \vec{x}, C\vec{y}}.
	\end{equation}
	Indeed, since the standard sesquilinear form $\sesqui$ is conjugate-linear in the second position, it follows from the fact that $C$ is conjugate-linear
	that $\blf$ is linear in both positions.
	
	It is not hard to see that the bilinear form \eqref{eq:BLF} is \emph{nondegenerate}, 
	in the sense that $[\vec{x},\vec{y}] = 0$ for all $\vec{y}$ in $\h$ if and only if $\vec{x} = \vec{0}$.
	We also have the Cauchy-Schwarz inequality
	\begin{equation*}
		|[\vec{x},\vec{y}]| \leq \norm{\vec{x}} \norm{\vec{y}},
	\end{equation*}
	which follows since $C$ is isometric.
	However, $\blf$ is not a true inner product since 
	$[e^{i\theta/2}\vec{x},e^{i\theta/2} \vec{x}] = e^{i\theta}[\vec{x},\vec{x}]$ for any $\theta$ and,
	moreover, it is possible for $[\vec{x},\vec{x}]=0$ to hold even it $\vec{x} \neq \vec{0}$.
	
	Two vectors $\vec{x}$ and $\vec{y}$ are \emph{$C$-orthogonal} if $[\vec{x},\vec{y}]=0$ (denoted by $\vec{x} \perp_C \vec{y}$).  We say 
	that two subspaces $\E_1$ and $\E_2$ are $C$-orthogonal (denoted $\E_1 \perp_C \E_2$)
	if $[\vec{x}_1,\vec{x}_2] = 0$ for every $\vec{x}_1$ in $\E_1$ and $\vec{x}_2$ in $\E_2$.	

	To a large extent, the study of complex symmetric operators is equivalent to the study of 
	symmetric bilinear forms.  Indeed, 
	for a fixed conjugation $C:\h\to\h$, there is a bijective
	correspondence between bounded, symmetric bilinear forms $B(x,y)$
	on $\h\times\h$ and bounded $C$-symmetric operators on $\h$.
	
	\begin{lemma}\label{LemmaRepresenting}
		If $B:\h\times\h \to\C$ is a bounded, bilinear form
		and $C$ is a conjugation on $\h$, then there exists a unique bounded linear operator
		$T$ on $\h$ such that
		\begin{equation}\label{EquationForm}
			B(\vec{x},\vec{y}) = [T\vec{x},\vec{y}],
		\end{equation}
		for all $\vec{x},\vec{y}$ in $\h$, where $\blf$ denotes the bilinear form \eqref{eq:BLF} corresponding to $C$.
		If $B$ is symmetric, then $T$ is $C$-symmetric. Conversely, a bounded
		$C$-symmetric operator $T$ gives rise to a bounded, symmetric bilinear
		form via \eqref{EquationForm}.
	\end{lemma}
	
	\begin{proof}
		If $B$ is a bounded, bilinear form, then $(\vec{x},\vec{y})\mapsto B(\vec{x},C\vec{y})$ defines a bounded,
		sesquilinear form.  Thus there
		exists a bounded linear operator $T:\h\to\h$ such that
		$B(\vec{x},C\vec{y}) = \inner{T\vec{x},\vec{y}}$ for all $\vec{x},\vec{y}$ in $\h$.
		Replacing $\vec{y}$ with $C\vec{y}$, we obtain
		$B(\vec{x},\vec{y}) = [T\vec{x},\vec{y}]$.  If $B(\vec{x},\vec{y}) = B(\vec{y},\vec{x})$, then
		$\inner{T\vec{y},C\vec{x}} = \inner{T\vec{x},C\vec{y}}$ so that
		$\inner{\vec{x},CT\vec{y}} = \inner{\vec{x},T^*C\vec{y}}$ holds for all $\vec{x},\vec{y}$.
		This shows that $CT=T^*C$ and hence $T$ is $C$-symmetric.
		Conversely, if $T$ is $C$-symmetric, then
		\begin{equation*}
			[T\vec{x},\vec{y}] = \inner{T\vec{x},C\vec{y}} = \inner{\vec{x},T^*C\vec{y}} = \inner{\vec{x},CT\vec{y}} = [\vec{x},T\vec{y}].
		\end{equation*}
		The isometric property of $C$ and the Cauchy-Schwarz
		inequality show that the bilinear form $[T\vec{x},\vec{y}]$
		is bounded whenever $T$ is.
	\end{proof}

	If $B$ is a given bounded bilinear form, then Lemma
	\ref{LemmaRepresenting} asserts that for each conjugation $C$ on
	$\h$, there exists a unique representing operator $T$ on $\h$, which is
	$C$-symmetric if $B$ is symmetric, such that
	$$B(\vec{x},\vec{y}) = \inner{\vec{x},CT\vec{y}}.$$
	Although the choice of $C$ is arbitrary, the conjugate-linear operator
	$CT$ is uniquely determined by the bilinear form $B(\vec{x},\vec{y})$. 
	One can also see that the positive operator $|T| = \sqrt{T^*T}$ is uniquely
	determined by the form $B$.  Indeed, since $B(\vec{x},\vec{y})=\inner{\vec{x},T^*C\vec{y}} = \inner{\vec{y},CT\vec{x}}$,
	the conjugate-linear operators $CT$ and $T^*C$ are intrinsic to $B$ and thus so is the positive
	operator $(T^*C)(CT) = T^*T = |T|^2$.
	
	Without any ambiguity, we say that
	a bounded bilinear form $B(\vec{x},\vec{y})$ is \emph{compact} if the modulus
	$|T|$ of any of the representing operators $T$ is compact. If
	$B(\vec{x},\vec{y})$ is a compact bilinear form, then the \emph{singular
	values} of $B$ are defined to be the eigenvalues of the positive
	operator $|T|$, repeated according to their multiplicity.

\section{Polar structure and singular values}

\subsection{The Godi\v{c}-Lucenko Theorem}
	It is well-known that any planar rotation can be obtained as the product of
	two reflections.  The following theorem of Godi\v{c} and Lucenko \cite{GL}
	generalizes this simple observation and provides an interesting
	perspective on the structure of unitary operators.

	\begin{theorem}\label{TheoremGL}
		If $U$ is a unitary operator on a Hilbert space $\h$, then there exist conjugations $C$ and $J$ on $\h$
		such that $U = CJ$ and $U^* = JC$.
	\end{theorem}
	
	The preceding theorem states that any unitary operator on a fixed Hilbert space can be 
	constructed by gluing together two copies of essentially the same conjugate-linear operator.
	Indeed, by Lemma \ref{LemmaAllSame} any conjugation on $\h$ can be represented as complex 
	conjugation with respect to a certain orthonormal basis.  In this sense, the conjugations 
	$C$ and $J$ in Theorem \ref{TheoremGL} are structurally identical objects.  Thus the fine structure of unitary operators 
	arises entirely in how two copies of the same object are put together.  
	The converse of Theorem \ref{TheoremGL} is also true.
	
	\begin{lemma}\label{LemmaCJ}
		If $C$ and $J$ are conjugations on a Hilbert space $\h$, then $U = CJ$ is a unitary operator.  Moreover, 
		$U$ is both $C$-symmetric and $J$-symmetric.
	\end{lemma}

	\begin{proof}
		If $U = CJ$, then (by the isometric property of $C$ and $J$) it follows that
		$\inner{f,U^*g} = \inner{Uf,g} = \inner{CJf,g} = \inner{Cg,Jf} = \inner{f,JCg}$
		for all $f,g$ in $\h$.  Thus $U^* = JC$ from which $U = CU^*C$ and $U = JU^*J$ both follow.  
	\end{proof}

	\begin{example}
		Let $U:\C^n \to \C^n$ be a unitary operator with $n$ (necessarily unimodular) 
		eigenvalues $\xi_1,\xi_2,\ldots,\xi_n$ and corresponding orthonormal
		eigenvectors $\vec{e}_1,\vec{e}_2,\ldots,\vec{e}_n$.  If $C$ and $J$ are defined by setting
		$C\vec{e}_k = \xi_k \vec{e}_k$ and $J\vec{e}_k = \vec{e}_k$ for $k = 1,2,\ldots, n$ and extending 
		by conjugate-linearly to all of $\C^n$, then
		clearly $U =CJ$.  By introducing offsetting unimodular parameters in the definitions of $C$ and $J$, 
		one sees that the Godi\v{c}-Lucenko decomposition of $U$ is not unique.
	\end{example}
			
	\begin{example}
		Let $\mu$ be a finite Borel measure on the unit circle $\T$.
		If $U$ denotes the unitary operator $[Uf](e^{i\theta}) = e^{i\theta}f(e^{i\theta})$ on $L^2(\T,\mu)$, then $U = CJ$ where
		$$[Cf](e^{i\theta}) = e^{i\theta}\overline{f(e^{i\theta})},\quad [Jf](e^{i\theta}) = \overline{f(e^{i\theta})}$$
		for all $f$ in $L^2(\T,\mu)$.  The proof of Theorem \ref{TheoremGL} follows from the spectral theorem
		and this simple example.
	\end{example}
	
	\begin{example}
		Let $\h = L^2(\R, dx)$ and let
		$$[\mathcal{F} f] (\xi) = \frac{1}{\sqrt{2 \pi}} \int_{\R} e^{-i x \xi} f(x) dx$$
		denote the Fourier transform.
		Since $[Jf](x) = \overline{f(x)}$ satisfies
		$\mathcal{F} = J\mathcal{F}^* J$, we see that $\mathcal{F}$ is a
		$J$-symmetric unitary operator. The Fourier transform
		is the product of two simple conjugations:
		$C = \mathcal{F}J$ is complex conjugation in the frequency domain and $J$ is complex
		conjugation in the state space domain.
	\end{example}

\subsection{Refined polar decomposition}
	The Godi\v{c}-Lucenko decomposition (Theorem \ref{TheoremGL}) can be generalized
	to complex symmetric operators.
	Recall that the \emph{polar decomposition} $T = U|T|$ of a bounded linear operator $T:\h\rightarrow\h$
	expresses $T$ uniquely as the product of a positive operator $|T| = \sqrt{T^*T}$ and a partial isometry $U$ that satisfies
	$\ker U = \ker |T|$ and which maps $(\ran|T|)^-$ onto $(\ran T)^-$.  
	The following lemma, whose proof we briefly sketch, is from \cite{CSO2}:

	\begin{theorem}\label{RPDB}
		If $T:\h\rightarrow\h$ is a bounded $C$-symmetric operator,
		then $T = CJ|T|$ where $J$ is a conjugation that commutes with $|T| = \sqrt{T^*T}$
		and all of its spectral projections.
	\end{theorem}

	\begin{proof}
		Write the polar decomposition $T = U|T|$ of $T$ and note that
		$T = CT^*C = (CU^*C)(CU|T|U^*C)$ since $U^*U$ is the orthogonal projection onto
		$(\ran|T|)^-$.  One shows that $\ker CU^*C = \ker CU|T|U^*C$,
		notes that $CU^*C$ is a partial isometry and that
		$CU|T|U^*C$ is positive, then concludes from the uniqueness
		of the terms in the polar decomposition that $U = CU^*C$
		(so that $U$ is $C$-symmetric) and that the conjugate-linear operator
		$CU = U^*C$ commutes with $|T|$ and hence with all of its spectral projections. One then verifies
		that this ``partial conjugation'' supported on $(\ran|T|)^-$ can be extended to
		a conjugation $J$ on all of $\h$.
	\end{proof}

	A direct application of the refined polar decomposition is an analogue of the celebrated Adamyan-Arov-Kre\u{\i}n theorem
	asserting that the optimal approximant of prescribed rank of a Hankel operator is also a Hankel operator (see 
	\cite{Peller} for complete details). The applications of the
	Adamyan-Arov-Kre\u{\i}n theorem to extremal problems of modern function theory are analyzed in a concise and definitive form in \cite{Sarason1}.
	The case of complex symmetric operators is completely parallel.	
	
	\begin{theorem}
		Let $T$ be a compact $C$-symmetric operator with singular values 
		$s_0 \geq s_1 \geq \cdots$, repeated according to multiplicity, then
		\begin{equation*}
			s_n = \inf_{ \substack{ \operatorname{rank} T' = n\\ \text{$T'$ $C$-symmetric}}} \norm{T - T'}.
		\end{equation*}
	\end{theorem}
	
	Some applications of this theorem to rational approximation (of Markov functions) in the complex plane are discussed
	in \cite{PutinarProkhorov}.

\subsection{Approximate antilinear eigenvalue problems}
	A new method for computing the norm and singular values of a complex symmetric operator
	was developed in \cite{CSO2, AAEPRI}.  This technique has been used to
	compute the spectrum of the modulus of a Foguel operator \cite{NMFO}
	and to study non-linear extremal problems arising in classical function theory \cite{NLEPHS}.

	Recall that Weyl's criterion \cite[Thm.~VII.12]{RS1} states that if $A$ is a bounded selfadjoint operator, then $\lambda \in \sigma(A)$
	if and only if there exists a sequence $\vec{x}_n$ of unit vectors so that $\lim_{n\rightarrow\infty} \norm{(A - \lambda I)\vec{x}_n} = 0$.
	The following theorem characterizes $\sigma(|T|)$ in terms of an
	\emph{approximate antilinear eigenvalue problem}.
	
	\begin{theorem}\label{TheoremSpectrum}
		Let $T$ be a bounded $C$-symmetric operator and write $T = CJ|T|$ where $J$ is a conjugation commuting
		with $|T|$ (see Theorem \ref{RPDB}).  If		
		$\lambda\geq 0$, then
		\begin{enumerate}\addtolength{\itemsep}{0.5\baselineskip}
			\item $\lambda$ belongs to $\sigma(|T|)$ 
			if and only if there exists a sequence of unit vectors $\vec{x}_n$ such that
			$$\lim_{n\rightarrow\infty} \norm{(T - \lambda C)\vec{x}_n} = 0.$$  Moreover, the $\vec{x}_n$ may be chosen so that $J\vec{x}_n = \vec{x}_n$ for all $n$.
			
			\item $\lambda$ is an eigenvalue of $|T|$ (i.e., a singular value of $T$) if and only if the antilinear eigenvalue problem
			$$T\vec{x} = \lambda C\vec{x}$$ has a nonzero solution $\vec{x}$.  Moreover, $\vec{x}$ may be chosen so that $J\vec{x} = \vec{x}$.
		\end{enumerate}
	\end{theorem}
	
	\begin{proof}
		Since the second statement follows easily from the first, 
		we prove only the first statement.  Following Theorem \ref{RPDB}, write
		$T = CJ|T|$ where $J$ is a conjugation that commutes with $|T|$.	
		By Weyl's criterion, $\lambda \geq 0$ belongs to $\sigma(|T|)$ if and only if there exists a sequence $\vec{u}_n$
		of unit vectors so that $\norm{\,|T|\vec{u}_n - \lambda \vec{u}_n} \rightarrow 0$.  Since $J$ is isometric and commutes with $|T|$, 
		this happens if and only if $\norm{\,|T|J\vec{u}_n - \lambda J \vec{u}_n} \rightarrow 0$ as well.
		Since not both of $\tfrac{1}{2}(\vec{u}_n + J\vec{u}_n)$ and $\tfrac{1}{2i}(\vec{u}_n - J\vec{u}_n)$ can be zero for a given $n$, we can obtain a
		sequence of unit vectors $\vec{x}_n$ such that $J\vec{x}_n = \vec{x}_n$ and 
		\begin{equation*}
		\norm{(T - \lambda C)\vec{x}_n} = \norm{CT\vec{x}_n - \lambda \vec{x}_n} = \norm{J|T|\vec{x}_n - \lambda \vec{x}_n} 
		= \norm{|T|\vec{x}_n - \lambda \vec{x}_n} \to 0.
		\end{equation*}
		On the other hand, if a sequence $\vec{x}_n$ satisfying the original criteria exists, then it follows from Theorem
		\ref{RPDB} that $\lim_{n\rightarrow\infty}\norm{ (|T| - \lambda I) \vec{x}_n} = 0$.  By Weyl's criterion, $\lambda \in \sigma(|T|)$.
	\end{proof}

\subsection{Variational principles}
	The most well-known result in the classical theory of complex symmetric matrices
	is the so-called \emph{Takagi factorization}.  However, as the authors of \cite[Sect.~3.0]{HJTopics} point out, 
	priority must be given to L.~Autonne, who published this theorem in 1915	\cite{Autonne}.
	
	\begin{theorem}\label{TheoremTakagi}
		If $A = A^T$ is $n \times n$, then there exists a unitary matrix $U$ such that $A = U \Sigma U^T$
		where $\Sigma = \diag(s_0,s_1,\ldots,s_{n-1})$ is the diagonal matrix that has 
		the singular values of $A$ listed along the main diagonal.
	\end{theorem}
		
	This result has been rediscovered many times, most notably by
	Hua in the study of automorphic functions \cite{Hua}, Siegel in symplectic geometry \cite{Siegel}, 
	Jacobson in projective geometry \cite{Jacobson}, and Takagi \cite{Takagi} in complex function theory.
	As a consequence of the Autonne-Takagi decomposition we see that
	\begin{equation*}
		\vec{x}^T A \vec{x} = \vec{x}^T U\Sigma U^T \vec{x} = (U^T \vec{x})^T \Sigma (U^T \vec{x}) = \vec{y}^T \Sigma \vec{y}
	\end{equation*}
	where $\vec{y} = U^T \vec{x}$.  This simple observation is the key to proving a complex symmetric analogue of
	the following important theorem,
	the finite dimensional \emph{minimax principle}. The general principle can be used to numerically compute
	the bound state energies for Schr\"odinger operators \cite[Thm.~XIII.1]{RS4}.
	
	\begin{theorem}
		If $A = A^*$ is $n \times n$, then for $0 \leq k \leq n-1$ the eigenvalues 
		$\lambda_0 \geq \lambda_1 \geq  \cdots \geq \lambda_{n-1}$ of $A$ satisfy
		\begin{equation*}
			 \min_{\codim \V = k} \quad\max_{\substack{ \vec{x} \in \V \\ \norm{ \vec{x} }=1}}  \vec{x}^* A \vec{x} = \lambda_{k}.
		\end{equation*}		
	\end{theorem}			

	The following analogue of minimax principle was discovered by J.~Danciger in 2006 \cite{Danciger}, while
	still an undergraduate at U.C.~Santa Barbara.

	\begin{theorem}
		If $A = A^T$ is $n \times n$, then the singular values
		$s_0 \geq s_1 \geq \cdots \geq s_{n-1}$ of $A$ satisfy
		\begin{equation*}
			\min_{\codim \V = k} \quad\max_{\substack{\vec{x} \in \V \\ 
			\norm{\vec{x}}=1}} \Re { \vec{x}^T\! A \vec{x}}  \quad=\quad 
			\begin{cases}
				s_{2k}	&	\text{if $0 \leq k < \frac{n}{2}$,} \\[5pt]
				0		&	\text{if $\frac{n}{2} \leq k \leq n$.}
			\end{cases}
		\end{equation*}
	\end{theorem}

	The preceding theorem is remarkable since the expression $\Re \vec{x}^T \!A \vec{x}$
	detects only the evenly indexed singular values.  The Hilbert space generalization of Danciger's minimax
	principle is the following \cite{VPSBF}.

	\begin{theorem}\label{TheoremMain}
		If $T$ is a compact $C$-symmetric operator on $\h$ and
		$\sigma_0 \geq \sigma_1 \geq \cdots \geq 0$
		are the singular values of $T$, then
		\begin{equation}\label{EquationCSMM}
			\min_{\codim\V=n} \max_{\substack{\vec{x} \in \V \\ \norm{\vec{x}}=1}} \Re [T\vec{x},\vec{x}]    =
			\begin{cases}
				\sigma_{2n} & \text{if $0 \leq n < \frac{\dim\h}{2}$},\\[5pt]
				0         & \text{otherwise}.
			\end{cases}
		\end{equation}
	\end{theorem}

	By considering the expression $\Re [T\vec{x},\vec{x}]$
	over $\R$-linear subspaces of $\h$, one avoids the ``skipping'' phenomenon and obtains all of the singular values of $T$.

	\begin{theorem}\label{TheoremMainReal}
		If $T$ is a compact $C$-symmetric operator on a separable
		Hilbert space $\h$ and
		$\sigma_0 \geq \sigma_1 \geq \cdots \geq 0$
		are the singular values of $T$, then
		\begin{align}
			\sigma_{n}  &= \min_{\codim_{\R}\V=n} \max_{\substack{\vec{x} \in \V \\ \norm{\vec{x}}=1}} \Re [T\vec{x},\vec{x}]  \label{EquationCSRMM}
		\end{align}
		holds whenever $0 \leq n < \dim\h$.  Here $\V$ ranges over all $\R$-linear subspaces
		of the complex Hilbert space $\h$ and $\codim_{\R}\V$ denotes the codimension of $\V$ in $\h$ when
		both are regarded as $\R$-linear spaces.
	\end{theorem}

	The proofs of these theorems do not actually
	require the compactness of $T$, only the discreteness
	of the spectrum of $|T|$.  It is therefore possible
	to apply these variational principles if one knows
	that the spectrum of $|T|$ is discrete.  Moreover, 
	these variational principles still apply to eigenvalues
	of $|T|$ that are located strictly above the essential spectrum of $|T|$.

\section{Spectral Theory}\label{Section:Spectral}

	Although the spectral theory of complex symmetric operators is still under development,
	we collect here a number of observations and basic results that are often sufficient 
	for analyzing specific examples.

\subsection{Direct sum decomposition} 
	The first step toward understanding a given operator is to resolve it, if possible, into an orthogonal direct sum of simpler
	operators.
	Recall that a bounded linear operator $T$ is called \emph{reducible} if $T \cong A \oplus B$ (orthogonal direct sum).  
	Otherwise, we say that $T$ is \emph{irreducible}.    
	An irreducible operator commutes with no orthogonal projections
	except for $0$ and $I$.  
	
	In low dimensions, every complex symmetric operator is a direct sum of irreducible complex symmetric operators
	\cite{UET}.

	\begin{theorem}\label{Theorem:ABCD}
		If $T:\h\to\h$ is a complex symmetric operator and $\dim \h \leq 5$, then
		$T$ is unitarily equivalent to a direct sum of irreducible complex symmetric operators.
	\end{theorem}

	The preceding theorem is false in dimensions six and above due to the following simple construction. 

	\begin{lemma}
		If $A:\h\to\h$ is a bounded linear operator and $C:\h\to\h$ is conjugation, then $T = A\oplus CA^*C$
		is complex symmetric.
	\end{lemma}
	
	\begin{proof}
		Verify that
		\begin{equation*}
			\minimatrix{A}{0}{0}{CA^*C} = \minimatrix{0}{C}{C}{0}\minimatrix{A}{0}{0}{CA^*C}^*\minimatrix{0}{C}{C}{0}.\qedhere
		\end{equation*}
	\end{proof}
	
	If $A$ is an irreducible operator that is \emph{not} complex symmetric, then $T = A\oplus CA^*C$ is a complex symmetric
	operator that possesses irreducible direct summands that are not complex symmetric.  In other words, the class of complex
	symmetric operators is \emph{not} closed under restriction to direct summands.  The correct
	generalization (in the finite dimensional case) of Theorem \ref{Theorem:ABCD} is the following \cite{UET}:

	\begin{theorem}
		If $T$ is a complex symmetric operator on a finite dimensional Hilbert space, then $T$ is unitarily equivalent to a direct sum of
		(some of the summands may be absent) of
		\begin{enumerate}\addtolength{\itemsep}{0.5\baselineskip}
			\item irreducible complex symmetric operators,
			\item operators of the form $A \oplus CA^*C$, where $A$ is irreducible and not a complex symmetric operator.
		\end{enumerate}
	\end{theorem}

The preceding result found unexpected applications to quantum computing, specifically to the trichotomy of constricted
quantum semigroups recently singled out by Singh \cite{Singh}, see also \cite{Ghosh-Singh}.

An operator is called \emph{completely reducible} if it does not admit any minimal reducing subspaces.  For instance, a normal operator
	is complete reducible if and only if it has no eigenvalues.
	In arbitrary dimensions, Guo and Zhu recently proved the following striking result \cite{GuoZhu}.

	\begin{theorem}
		If $T$ is a bounded complex symmetric operator on a Hilbert space, then $T$ is unitarily equivalent to a direct sum 
		(some of the summands may be absent) of
		\begin{enumerate}\addtolength{\itemsep}{0.5\baselineskip}
			\item completely reducible complex symmetric operators,
			\item irreducible complex symmetric operators,
			\item operators of the form $A \oplus CA^*C$, where $A$ is irreducible and not a complex symmetric operator.
		\end{enumerate}
	\end{theorem}

	A related question, of interest in matrix theory, is whether a matrix $A$ that is unitarily equivalent to $A^T$ is complex symmetric.   
	This conjecture holds for matrices that are $7 \times 7$ smaller, but fails for matrices that are $8 \times 8$ or larger \cite{UET}.

	Currently, the preceding theorems are the best available.  It is not yet clear whether a concrete functional model
	for, say, irreducible complex symmetric operators, can be obtained.  However, a growing body of evidence suggests
	that truncated Toeplitz operators may play a key role (see the survey article \cite{RPTTO}).

\subsection{$C$-projections}	

	If $T$ is a bounded linear operator and $f$ is a holomorphic function on a (not necessarily 
	connected) neighborhood $\Omega$ of $\sigma(T)$, then the Riesz functional calculus allows us to define
	an operator $f(T)$ via the Cauchy-type integral
	\begin{equation}\label{eq-ResolventIntegral}
		f(T) = \frac{1}{2\pi i } \int_{\Gamma} f(z) (zI - T)^{-1} dz
	\end{equation}
	in which $\Gamma$ denotes a finite system of rectifiable Jordan curves, oriented in the positive sense and lying in $\Omega$ \cite[p.568]{Dunford1}.

	For each \emph{clopen} (relatively open and closed) subset $\Delta$ of $\sigma(T)$,
	there exists a natural idempotent $P(\Delta)$ defined by the formula
	\begin{equation}\label{eq-RieszIdempotent}
		P(\Delta) = \frac{1}{2\pi i } \int_{\Gamma} (zI - T)^{-1}\,dz
	\end{equation}
	where $\Gamma$ is any rectifiable Jordan curve such that $\Delta$ is contained in the interior $\operatorname{int} \Gamma$ of $\Gamma$ and
	$\sigma(T)\backslash\Delta$ does not intersect $\operatorname{int} \Gamma$.  We refer to this
	idempotent as the \emph{Riesz idempotent} corresponding to $\Delta$.

	If the spectrum of an operator $T$ decomposes as the disjoint union of two clopen sets,
	then the corresponding Riesz idempotents are usually not orthogonal projections.  Nevertheless, the Riesz idempotents that arise from 
	complex symmetric operators have some nice features.

	\begin{theorem}\label{TheoremRiesz}
		Let $T$ be a $C$-symmetric operator.
		If $\sigma(T)$ decomposes as the disjoint union $\sigma(T) = \Delta_1 \cup \Delta_2$ of two clopen sets, then 
		the corresponding Riesz idempotents $P_1 =P(\Delta_1)$ and $P_2 = P(\Delta_2)$ defined by \eqref{eq-RieszIdempotent} are
		\begin{enumerate}\addtolength{\itemsep}{0.5\baselineskip}
			\item $C$-symmetric:  $P_i = CP_i^*C$ for $i=1,2$,
			
			\item $C$-orthogonal, in the sense that $\ran P_1 \perp_C \ran P_2$.
		\end{enumerate}
	\end{theorem}
	
	The proof relies on the fact that the resolvent $(zI - T)^{-1}$ is $C$-symmetric for all $z \in \C$.
	We refer to a $C$-symmetric idempotent as a \emph{$C$-projection}.
	In other words, a bounded linear operator $P$ is a $C$-projection if and only if $P = CP^*C$ and $P^2 = P$.
	It is not hard to see that if $P$ is a $C$-projection, then $\norm{P} \geq 1$ and $\ran P$ is closed.  Moreover, for any $C$-projection, we have
	$\ker P \cap \ran P = \{0\}$.
	This is not true for arbitrary complex symmetric operators (e.g., a $2 \times 2$ nilpotent Jordan matrix).

	A classical theorem of spectral theory \cite[p.579]{Dunford1} states that if $T$ is a compact operator, 
	then every nonzero point $\lambda$ in $\sigma(T)$ is an eigenvalue of finite order $m = m(\lambda)$.  
	For each such $\lambda$, the corresponding Riesz idempotent
	has a nonzero finite dimensional range given by
	$\ran P_{\lambda} = \ker (T - \lambda I)^{m}$.
	In particular, the nonzero elements of the spectrum of a compact operator correspond to generalized eigenspaces.

	\begin{theorem}\label{TheoremCompact}
		The generalized eigenspaces of a compact $C$-symmetric operator are $C$-orthogonal.
	\end{theorem}
	
	\begin{proof}
		It follows immediately from Theorem \ref{TheoremRiesz} and the preceding remarks that the generalized
		eigenspaces corresponding to nonzero eigenvalues of a compact $C$-symmetric operator $T$ are mutually $C$-orthogonal.
		Since $0$ is the only possible accumulation point of the eigenvalues of $T$, it follows that a generalized eigenvector
		corresponding to a nonzero eigenvalue is $C$-orthogonal to any vector in the range of
		\begin{equation*}
			P_{\epsilon} = \frac{1}{2\pi i } \int_{|z| = \epsilon} (zI - T)^{-1}\,dz
		\end{equation*}
		if $\epsilon>0$ is taken sufficiently small.  In particular, $\ran P_{\epsilon}$ contains the generalized eigenvectors
		for the eigenvalue $0$ (if any exist).
	\end{proof}

\subsection{Eigenstructure}\label{Subsection:Eigenstructure}
	With respect to the bilinear form $\blf$, it turns out that $C$-symmetric operators superficially resemble selfadjoint operators.
	For instance, an operator $T$ is $C$-symmetric if and only if
	$[T\vec{x},\vec{y}] = [\vec{x},T\vec{y}]$ for all $\vec{x},\vec{y}$ in $\h$.  
	As another example, the eigenvectors of a
	$C$-symmetric operator corresponding to distinct eigenvalues
	are orthogonal with respect to $\blf$, even though they are not
	necessarily orthogonal with respect to the original sesquilinear form $\sesqui$.  
	
	\begin{lemma}\label{LemmaEigenvectors}
		The eigenvectors of a $C$-symmetric operator $T$ corresponding to distinct eigenvalues are
		orthogonal with respect to the bilinear form $\blf$.
	\end{lemma}
	
	\begin{proof}
		The proof is essentially identical to the corresponding proof for selfadjoint operators.
		If $\lambda_1 \neq \lambda_2$, $T\vec{x}_1 = \lambda_1 \vec{x}_1$, and $T\vec{x}_2 = \lambda_2 \vec{x}_2$, then
		\begin{equation*}
			\lambda_1 [\vec{x}_1,\vec{x}_2] = [\lambda_1 \vec{x}_1,\vec{x}_2] 
			= [T\vec{x}_1,\vec{x}_2]=[\vec{x}_1,T\vec{x}_2] = [\vec{x}_1,\lambda_2 \vec{x}_2] = \lambda_2 [\vec{x}_1,\vec{x}_2].
		\end{equation*}
		Since $\lambda_1 \neq \lambda_2$, it follows that $[\vec{x}_1,\vec{x}_2]=0$.  
	\end{proof}

	There are some obvious differences between selfadjoint and complex symmetric operators.
	For instance, a complex symmetric matrix can have any possible Jordan canonical form (Theorem \ref{TheoremFolk})
	whereas a selfadjoint matrix must be unitarily diagonalizable.
	The following result shows that complex symmetric operators 
	have a great deal more algebraic structure than one can expect from an arbitrary operator
	(see \cite{ESCSO} for a complete proof; Theorem \ref{TheoremCompact}
	addresses only the compact case).

	\begin{theorem}\label{TheoremGeneralized}
		The generalized eigenspaces of a $C$-symmetric operator corresponding to distinct eigenvalues 
		are mutually $C$-orthogonal.
	\end{theorem}

	We say that a vector $x$ is \emph{isotropic} if $[\vec{x},\vec{x}] = 0$.
	Although $\vec{0}$ is an isotropic vector, 
	nonzero isotropic vectors are nearly unavoidable (see Lemma \ref{LemmaIsotropic} below).
	However, isotropic eigenvectors often have 
	meaningful interpretations.  For example, isotropic eigenvectors of complex symmetric matrices
	are considered in \cite{Scott} in the context of elastic wave propagation.  In that theory,
	isotropic eigenvectors correspond to circularly polarized waves.  
	
	The following simple lemma hints at the relationship
	between isotropy and multiplicity that we will explore later.
	
	\begin{lemma}\label{LemmaIsotropic}
		If $C:\h\to\h$ is a conjugation, then
		every subspace of dimension $\geq 2$ contains isotropic vectors for the bilinear form $\blf$.
	\end{lemma}
	
	\begin{proof}
		Consider the span of
		two linearly independent vectors $\vec{x}_1$ and $\vec{x}_2$.  If $\vec{x}_1$ or $\vec{x}_2$ is 
		isotropic, we are done.  If neither $\vec{x}_1$ nor $\vec{x}_2$ is isotropic, then 		\begin{equation*}
			\vec{y}_1 = \vec{x}_1, \quad \vec{y}_2 = \vec{x}_2 - \frac{[\vec{x}_2,\vec{x}_1]}{ [\vec{x}_1,\vec{x}_1] }\vec{x}_1
		\end{equation*}
		are $C$-orthogonal and have the same span as $\vec{x}_1,\vec{x}_2$.
		In this case, either $\vec{y}_2$ is isotropic (and we are done) or neither $\vec{y}_1$ nor $\vec{y}_2$ is isotropic.
		If the latter happens, we may assume that $\vec{y}_1$ and $\vec{y}_2$ satisfy $[\vec{y}_1,\vec{y}_1] = [\vec{y}_2,\vec{y}_2] = 1$.
		Then the vectors $\vec{y}_1 \pm i \vec{y}_2$ are both isotropic and have the same span as
		$\vec{x}_1$ and $\vec{x}_2$.
	\end{proof}	

	The following result shows that the existence of an isotropic eigenvector for an \emph{isolated} eigenvalue is determined
	by the multiplicity of the eigenvalue.
	
	\begin{theorem}\label{TheoremIsotropic}
		If $T$ is a $C$-symmetric operator, then an isolated eigenvalue 
		$\lambda$ of $T$ is simple if and only if $T$ has no isotropic eigenvectors for $\lambda$.
	\end{theorem}

	\begin{proof}
		If $\lambda$ is an isolated eigenvalue of $T$, then the Riesz idempotent $P$ corresponding to $\lambda$
		is a $C$-projection.  If $\lambda$ is a simple eigenvalue, then the eigenspace corresponding to $\lambda$
		is spanned by a single unit vector $\vec{x}$.  If $\vec{x}$ is isotropic, then it is $C$-orthogonal to all of $\h$
		since $\vec{x}$ is $C$-orthogonal to the range of the complementary $C$-projection $I - P$.  This would imply that
		$\vec{x}$ is $C$-orthogonal to all of $\h$ and hence $\vec{x} = \vec{0}$, a contradiction.

		If $\lambda$ is not a simple eigenvalue, then there are two cases to consider.
		\medskip
		  
		\noindent\textsc{Case 1}: If $\dim \ker(T - \lambda I) > 1$, then by Lemma \ref{LemmaIsotropic},
		$\ker(T - \lambda I)$ contains an isotropic vector.  Thus $T$ has an isotropic eigenvector corresponding to the eigenvalue $\lambda$.
		\medskip
		
		\noindent\textsc{Case 2}: If $\dim \ker(T - \lambda I) = 1$, then $\ker(T - \lambda I) = \operatorname{span}\{\vec{x}\}$ for some $\vec{x} \neq \vec{0}$ 
		and $\dim \ker(T - \lambda I)^2 > 1$ since $\lambda$ is not a simple eigenvalue.  
		We can therefore find a nonzero generalized eigenvector $\vec{y}$ for $\lambda$ such that
		$\vec{x} = (T - \lambda I)\vec{y}$.  Thus
		\begin{equation*}
			[\vec{x},\vec{x}] = [\vec{x}, (T - \lambda I)\vec{y}] = [(T - \lambda I)\vec{x},\vec{y}] = [\vec{0},\vec{y}] = 0
		\end{equation*}
		and hence $\vec{x}$ is an isotropic eigenvector.
	\end{proof}

	\begin{example}
		The hypothesis that $\lambda$ is an isolated eigenvalue is crucial.
		The operator $S \oplus S^*$, where $S$ is the unilateral shift on $\ell^2(\N)$, is complex symmetric  
		and has each point in the open unit disk as a simple eigenvalue \cite{CSO2}.  Nevertheless,
		every eigenvector is isotopic.
	\end{example}
	
\subsection{$C$-orthonormal systems and Riesz bases}\label{Subsection:Diagonalization}

Let $\h$ be a separable, infinite dimensional complex Hilbert space
endowed with a conjugation $C$.
Suppose that $\{\vec{u}_n\}$ is a complete system of \emph{$C$-orthonormal} vectors:
\begin{equation}\label{EquationOrthogonality}
  [\vec{u}_n,\vec{u}_m] = \delta_{nm},
\end{equation}
in which $\blf$ denotes the symmetric bilinear form \eqref{eq:BLF} 
induced by $C$.  In other words, suppose that
$\{\vec{u}_n\}$ and $\{C\vec{u}_n\}$ are
complete biorthogonal sequences in $\h$.  Such sequences frequently arise
as the eigenvectors for a $C$-symmetric operator (see Subsection \ref{Subsection:Eigenstructure}).
Most of the following material originates in \cite{ESCSO}.

We say that a vector $\vec{x}$ in $\h$ is \emph{finitely supported} if it is a finite linear combination of the $\vec{u}_n$
and we denote the linear manifold of finitely supported vectors by $\mathcal{F}$.
Due to the $C$-orthonormality of the $\vec{u}_n$, it follows immediately that each such $\vec{x} \in \mathcal{F}$ can be recovered via the
\emph{skew Fourier expansion}
\begin{equation}\label{eq-SkewFourier}
  \vec{x} = \sum_{n=1}^{\infty}\, [\vec{x},\vec{u}_n] \vec{u}_n,
\end{equation}
where all but finitely many of the coefficients $[\vec{x},\vec{u}_n]$ are nonzero.
We will let $A_0:\mathcal{F} \rightarrow \h$ denote the linear extension of the map 
$A_0 u_n = Cu_n$ to $\mathcal{F}$.  Since $\mathcal{F}$ is a dense linear submanifold of $\h$,
it follows that if $A_0:\mathcal{F}\rightarrow\h$ is bounded on $\mathcal{F}$, then $A_0$ has a unique
bounded extension (which we denote by $A$) to all of $\h$.

It turns out that the presence of the conjugation $C$ ensures that such an extension must have
several desirable algebraic properties.  In particular, the following lemma shows that if $A$ is bounded,
then it is $C$-orthogonal.  Specifically, we say that an operator $U:\h\rightarrow\h$ is \emph{$C$-orthogonal} if
$CU^*CU = I$.  The terminology comes from the fact that, when represented with respect to a $C$-real orthonormal basis,
the corresponding matrix will be complex orthogonal (i.e., $U^T U = I$ as matrices).  

The importance of $C$-orthogonal operators lies in the fact that they preserve the bilinear form induced by $C$.  To be specific,
$U$ is a $C$-orthogonal operator if and only if  $[U\vec{x},U\vec{y}] = [\vec{x},\vec{y}]$ for all $\vec{x},\vec{y}$ in $\h$.  Unlike unitary operators, $C$-orthogonal
operators can have arbitrarily large norms.  In fact, unbounded $C$-orthogonal operators are considered in \cite{Riss},
where they are called \emph{$J$-unitary} operators.

\begin{lemma}\label{LemmaOrthogonal}
  If $A_0$ is bounded, then its extension $A:\h\rightarrow\h$ is positive and $C$-orthogonal.  If this is the case, then $A$ is invertible
  with $A^{-1} = CAC \geq 0$ and the operator $B = \sqrt{A}$ is also $C$-orthogonal.
\end{lemma}

\begin{proof}
  By \eqref{eq-SkewFourier}, it follows that
 $\inner{A_0 \vec{x},\vec{x}} = \sum_{n=1}^{\infty} |[\vec{x},\vec{u}_n]|^2 \geq 0$
  for all $\vec{x}$ in $\mathcal{F}$.  If $A_0$ is bounded, then it follows by continuity that $A$ will be positive.  
  The fact that $A$   is $C$-orthogonal (hence invertible) follows from the
  fact that  $(CA^*C) A \vec{u}_n = (CA)^2 \vec{u}_n = \vec{u}_n$
  for all $n$.
  Since $(CBC)(CBC) = CAC = A^{-1}$ and $CBC \geq 0$,
  it follows that $CBC$ is a positive square
  root of $A^{-1}$.  By the uniqueness of the positive square root of a positive operator, we see that
  $CBC = B^{-1}$ and hence $B$ is also $C$-orthogonal. 
\end{proof}

We remark that Lemma \ref{LemmaOrthogonal} shows that if the map $\vec{u}_n \mapsto C\vec{u}_n$ is bounded, then its
linear extension $A:\h\rightarrow\h$ is necessarily invertible.
This property distinguishes $C$-orthonormal systems $\{\vec{u}_n\} $ and their duals $\{C\vec{u}_n\}$ from general
biorthogonal systems.
Among other things, Lemma \ref{LemmaOrthogonal} also shows that if $A_0$ is bounded, then the \emph{skew conjugation}
$J\left(\sum_{n=1}^{\infty} c_n \vec{u}_n\right) = \sum_{n=1}^{\infty} \overline{c_n}\vec{u}_n$
(defined initially on $\mathcal{F}$) is given by
\begin{equation*}
  J = CA = CBB = B^{-1}CB.
\end{equation*}
In other words, the skew conjugation $J$ is similar to our original conjugation $C$ via
the operator $B = \sqrt{A}$.
Another consequence of the boundedness of $A_0$
is the existence of a natural orthonormal basis
for $\h$.

 \begin{lemma}
  If $A_0$ is bounded, then the vectors $\{\vec{s}_n\} $ defined by $\vec{s}_n = B\vec{u}_n$ (where $B = \sqrt{A}$) satisfy the following:
  \begin{enumerate}\addtolength{\itemsep}{0.5\baselineskip}
    \item $\{\vec{s}_n\} $ is orthonormal:  $\inner{\vec{s}_j,\vec{s}_k} = \delta_{jk}$ for all $j,k$,

    \item $\{\vec{s}_n\} $ is $C$-orthonormal:  $[\vec{s}_j,\vec{s}_k] = \delta_{jk}$ for all $j,k$,

    \item $C\vec{s}_n = \vec{s}_n$ for all $n$.
  \end{enumerate}
  Furthermore, $\{\vec{s}_n\} $ is an orthonormal basis for $\h$.
\end{lemma}

\begin{proof}
  This follows from direct computations:
  \begin{align*}
    \inner{\vec{s}_j,\vec{s}_k} &= \inner{B\vec{u}_j, B\vec{u}_k}
            = \inner{\vec{u}_j,A\vec{u}_k} =  \inner{\vec{u}_j,C\vec{u}_k}
            = [\vec{u}_j,\vec{u}_k] = \delta_{jk},\\
    [\vec{s}_j,\vec{s}_k] &= \inner{\vec{s}_j, C\vec{s}_k}
        = \inner{B\vec{u}_j, CB \vec{u}_k} = \inner{B\vec{u}_j, B^{-1}C\vec{u}_k}
        = \inner{\vec{u}_j, C \vec{u}_k}
        = \delta_{jk},\\
    C\vec{s}_j &= CB\vec{u}_j = B^{-1} C\vec{u}_j = B^{-1} B^2 \vec{u}_j = B\vec{u}_j = \vec{s}_j.
  \end{align*}
  We now show that the system $\{\vec{s}_n\} $ is complete.
  If $\vec{x}$ is orthogonal to each $\vec{s}_j$, then
  $\inner{B\vec{x},\vec{u}_j} = \inner{\vec{x},B\vec{u}_j} = \inner{\vec{x},\vec{s}_j} = 0$ for all $j$.
  Since $B$ is invertible, it follows that $\vec{x} = \vec{0}$ since $\{\vec{u}_n\} $ is complete.
\end{proof}

If the operator $A_0$ is bounded, then its extension
$A$ is a positive, invertible operator whose spectrum
is bounded away from zero.  Thus
$\Theta = -i \log A$
can be defined using the functional calculus for $A$ and
the principal branch of the logarithm.  Since $A$ is
selfadjoint and the principal branch of the logarithm
is real on $(0,\infty)$, it follows that $\Theta$ is
skew-Hermitian:
$\Theta^* = -\Theta$.
Moreover, since $A$ is a $C$-orthogonal operator, it follows that $\Theta$ is a \emph{$C$-real} operator:
$\overline{\Theta} = \Theta$,
where $\overline{\Theta} = C\Theta C$.

Returning to our original $C$-symmetric operator $T$, we see that if $A_0$ is bounded, then $T$ is similar to the 
diagonal operator $D:\h\rightarrow\h$ defined by $D\vec{s}_n = \lambda_n \vec{s}_n$ since
$T = B^{-1} D B$.
Writing this in terms of the exponential representation
$A = \exp(i\Theta)$ and inserting a parameter $\tau\in[0,1]$, we obtain
a family of operators
$$T_{\tau} = e^{ -\frac{i\tau}{2}\Theta} D e^{ \frac{i\tau}{2} \Theta}$$
that satisfies $T_0 = D$ and $T_1 = T$.  This
provides a continuous deformation of $T$ to its diagonal model $D$.
We also remark that the fact that $\Theta$ is $C$-real and skew-Hermitian
implies that the operators $\exp(\pm\frac{i\tau}{2}\Theta)$ are
$C$-orthogonal for all $\tau$.  From here, it is easy to show that
each intermediate operator $T_{\tau}$ is $C$-symmetric and that
the path $\tau \mapsto T_{\tau}$ from $[0,1]$ to $B(\h)$ is norm continuous.

The following theorem provides a number of conditions equivalent to the boundedness of $A_0$:

\begin{theorem}\label{TheoremTFAE}
  If $\{\vec{u}_n\}$ is a complete $C$-orthonormal system in $\h$, then
  the following are equivalent:
  \begin{enumerate}\addtolength{\itemsep}{0.5\baselineskip}
    \item $\{\vec{u}_n\}$ is a Bessel sequence with Bessel bound $M$,

    \item $\{\vec{u}_n\}$ is a Riesz basis with lower and upper bounds $M^{-1}$ and $M$,

    \item $A_0$ extends to a bounded linear operator on $\h$
      satisfying $\norm{A_0} \leq M$,

    \item There exists $M>0$ satisfying:
      $$\norm{\sum_{n=1}^N \overline{c_n}  \vec{u}_n} \leq M
      \norm{\sum_{n=1}^N c_n \vec{u}_n},$$
      for every finite sequence $c_1,c_2,\ldots,c_N$.

    \item The Gram matrix $(\inner{\vec{u}_j,\vec{u}_k})_{j,k=1}^\infty$
      dominates its transpose:
      $$\big(M^2 \inner{\vec{u}_j,\vec{u}_k} - \inner{\vec{u}_k,\vec{u}_j} \big)_{j,k=1}^\infty \geq 0$$
      for some $M > 0$.

    \item The Gram matrix $G = (\inner{\vec{u}_j,\vec{u}_k})_{j,k=1}^\infty$ is
      bounded on $\ell^2(\N)$
      and orthogonal ($G^T G = I$ as matrices).  Furthermore, $\norm{G} \leq M$

    \item The skew Fourier expansion $$\sum_{n=1}^{\infty}\, [f,\vec{u}_n]\vec{u}_n$$ converges
      in norm for each $f \in H$ and
      $$\frac{1}{M} \norm{f}^2 \leq \sum_{n=1}^{\infty} |[f,\vec{u}_n]|^2 \leq M \norm{f}^2.$$
  \end{enumerate}
  In all cases, the infimum over all such $M$ equals the norm of $A_0$.
\end{theorem}

A nontrivial application of the preceding result to free interpolation in the Hardy space of the unit disk
is described in \cite{ICS}. More appropriate for the profile of the present survey are the following
Riesz basis criteria for the eigenvectors of a complex symmetric operator.

A classical observation due to Glazman \cite{Glazman2} gives
conditions solely in terms of the (simple) spectrum of a
dissipative operator for the root vectors to form a Riesz basis
\cite{Glazman2}. This idea was further exploited, and put into a
general context in the last chapter of Gohberg and Kre\u{\i}n's
monograph \cite{GK1}.
We illustrate below how complex symmetry can be used to weaken
Glazman's assumption without changing the conclusion.

Suppose that $T$ is a $C$-symmetric contraction with a complete
system $\{\vec{u}_n\}$ of eigenvectors corresponding to the simple
eigenvalues $\{\lambda_n\}$. Remark that, due to the
$C$-symmetry assumption $[\vec{u}_n, \vec{u}_m] = 0$ for $n \neq m$ (Lemma \ref{LemmaEigenvectors}). Moreover,
$[\vec{u}_n,\vec{u}_n]\neq 0$ because the system $\{\vec{u}_n\}$ is complete.

Letting $D = I - T^*T$, we see that $D\geq 0$ and hence $\inner{D\vec{x},\vec{y}}$ defines a positive sesquilinear form
on $\h \times \h$ and thus
$$|\inner{D\vec{x},\vec{y}} | \leq \sqrt{\inner{D\vec{x},\vec{x}}}\sqrt{\inner{D\vec{y},\vec{y}}}$$
for all $\vec{x},\vec{y}$ in $\h$.  Setting $\vec{x} = \vec{u}_j$ and $\vec{y} = \vec{u}_k$
we find that
\begin{align*}
  |\inner{D\vec{u}_j,\vec{u}_k}|
  &= |\inner{\vec{u}_j,\vec{u}_k} - \inner{T\vec{u}_j,T\vec{u}_k}|\\
  &= |1 - \lambda_j \overline{\lambda_k} ||\inner{\vec{u}_j,\vec{u}_k}|.
\end{align*}
Similarly, we find that
\begin{equation*}
  \sqrt{\inner{D\vec{u}_j,\vec{u}_j}}
  = \sqrt{ 1 - |\lambda_j|^2}\norm{\vec{u}_j}
\end{equation*}
and thus
$$|\inner{\vec{u}_j,\vec{u}_k}| \leq \norm{\vec{u}_j}\norm{\vec{u}_k}\frac{ \sqrt{1 - |\lambda_j|^2}\sqrt{1 - |\lambda_k|^2}}
{|1 - \lambda_j \overline{\lambda_k}|}.$$
This leads us to the following result from \cite{ICS}.

\begin{theorem} Let $T$ be a contractive $C$-symmetric operator
with simple spectrum $\{\lambda_n\}_{n=1}^\infty$ and complete
system of corresponding eigenvectors $\{\vec{u}_n\}$. Assume
that the normalization $[\vec{u}_n,\vec{u}_n] = 1, \ n \geq 1,$ is adopted. If
the matrix
$$\left[ \norm{\vec{u}_j}\norm{\vec{u}_k}\frac{ \sqrt{1 - |\lambda_j|^2}\sqrt{1 - |\lambda_k|^2}}
{|1 - \lambda_j \overline{\lambda_k}|} \right]_{j,k=1}^{\infty}$$ defines a linear bounded
operator on $\ell^2(\N)$, then $\{\vec{u}_n\}$ is a Riesz basis for $\h$.
\end{theorem}

Glazman's original result \cite{Glazman2}, stated for unit eigenvectors and
without the complex symmetry assumption, invoked the finiteness of
the Hilbert-Schmidt norm of the matrix
$$\frac{ \sqrt{1 - |\lambda_j|^2}\sqrt{1 - |\lambda_k|^2}} {|1 -\lambda_j \overline{\lambda_k}|}.$$ 

A completely analogous result can be stated for an unbounded
$C$-symmetric purely dissipative operator \cite{ICS}.

\begin{theorem}
  Let $T:\mathcal{D}\to \h$ be a $C$-symmetric, pure dissipative operator
  with simple spectrum $\{\lambda_n\}$ and complete
  sequence of corresponding unit eigenvectors $\{\vec{v}_n \}$.
  If the separation condition
  \begin{equation}\label{Separation}
  	\inf_n |[\vec{v}_n ,\vec{v}_n ]|  > 0
  \end{equation}
  holds and if the matrix
  \begin{equation}\label{D-Carleson}
    \left[  \frac{ \sqrt{ (\Im \lambda_j)(\Im \lambda_k)}}{ | \lambda_j - \overline{\lambda_k}|}    \right]_{j,k=1}^{\infty}
  \end{equation}
  defines a bounded linear operator on $\ell^2(\N)$, then
  \begin{enumerate}\addtolength{\itemsep}{0.5\baselineskip}
    \item The sequence $\{\vec{v}_n \}$ forms a Riesz basis for $\h$.

    \item Each $\vec{x}$ in $\h$ can be represented by a norm-convergent
      skew Fourier expansion given by
      $$\vec{x} = \sum_{n=1}^{\infty} \frac{[\vec{x},\vec{v}_n ]}{[\vec{v}_n ,\vec{v}_n ]} \vec{v}_n .$$
  \end{enumerate}
  In particular, if the matrix \eqref{D-Carleson} is bounded above,
  then it is also invertible.
\end{theorem}

We close this section with two instructive examples.

\begin{example} 
Let $\h = L^2[-\pi,\pi]$, endowed with normalized Lebesgue measure $dm= \frac{dt}{2\pi}$,
and let $[Cf] (x) = \overline{f(-x)}$.
Let $h$ be an odd, real-valued measurable function on $[-\pi,\pi]$, such that $e^h$ is unbounded but belongs to $\h$. 
 The vectors
$$ u_n(x) = \exp (h(x) + inx), \ \ n \in \Z,$$
are uniformly bounded in norm since $\| \vec{u}_n\| = \| e^h\|$ and $C$-orthonormal.
Since the operator $A_0$ is simply multiplication by $e^{-2h}$,
it is essentially selfadjoint and unbounded. Thus $\{u_n\}$ is not a Riesz basis, in spite of the fact that
it is a $C$-orthonormal system whose vectors are uniformly bounded in norm.
\end{example}

\begin{example}
  Let $w = \alpha+i\beta$ where $\alpha$ and $\beta$ are real constants and consider
  $L^2[0,1]$, endowed with the conjugation $[Cf](x) = \overline{f(1-x)}$.  A short
  computation shows that if $w$ is not an integer multiple of $2\pi$, then the vectors
  \begin{equation*}
    u_n(x) = \exp[i(w + 2\pi n)(x - \tfrac{1}{2})], \quad n\in\mathbb{Z},
  \end{equation*}
  are eigenfunctions of the $C$-symmetric operator
  \begin{equation*}
    [Tf](x) = e^{iw/2} \int_0^x f(y)\,dy + e^{-iw/2} \int_x^1 f(y)\,dy
  \end{equation*}
  (i.e., $T  = e^{iw/2} V + e^{-iw/2}V^*$ where $V$ denotes the Volterra integration operator; see Example \ref{Example:Volterra})
  and that the system $\{u_n\}$ is complete and $C$-orthonormal.  One the other hand,
  one might also say that the $u_n$ are eigenfunctions of the derivative operator with 
  boundary condition $f(1) = e^{iw} f(0)$.

  We also see that the map $A_0$ given by $u_n \mapsto Cu_n$ extends to a bounded operator on all of $L^2[0,1]$.
  Indeed, this extension is simply the multiplication operator
  $[Af](x) = e^{2\beta(x- 1/2)}f(x)$
  whence $B = \sqrt{A}$ is given by
  \begin{equation*}
    [Bf](x) = e^{\beta(x- 1/2)}f(x).
  \end{equation*}
  The positive operators $A$ and $B$ are both $C$-orthogonal (i.e., $CA^*CA = I$ and $CB^*CB = I$) and the system
  $\{u_n\}$ forms a Riesz basis for $L^2[0,1]$.  In fact, $\{u_n\}$ is the image of the $C$-real 
  orthonormal basis $\{s_n\}$, defined by $s_n = Bu_n$, under the bounded and invertible operator $B^{-1}$.
  The $s_n$ are given by 
  \begin{equation*}
    s_n(x) = \exp[i(\alpha + 2\pi n)(x - \tfrac{1}{2})]
  \end{equation*}
  and they are easily seen to be both orthonormal and $C$-real \cite[Lem.~4.3]{CCO}.
\end{example}

\subsection{Local spectral theory} 
	Among all of the various relaxations of the spectral properties of a normal operator, Foia\c{s}' notion of decomposability
	is one of the most general and versatile. A bounded linear operator $T:\h\to\h$ is called \emph{decomposable} if for every finite 
	open cover of its spectrum
	$$ \sigma(T) \subseteq U_1 \cup U_2 \cup \ldots \cup U_n,$$ there exists 
	closed $T$-invariant subspaces $\h_1,\h_2,\ldots,\h_n$, with the property that
	$$\qquad\qquad\sigma(T|_{\h_i}) \subseteq U_i, \qquad  1 \leq i \leq n$$ 
	and
	$$ \h_1 + \h_2 +\cdots +\h_n = \h.$$
	Checking for decomposability based upon the definition is highly nontrivial.
	In this respect, the early works of Dunford and Bishop are notable for
	providing simple decomposability criteria. We only mention Bishop's property $(\beta)$:  for every open set $U \subseteq \C$, the map
	$$ zI-T: \mathcal{O}(U,\h) \to \mathcal{O}(U,\h),$$
	is injective and has closed range.  Here $\mathcal{O}(U,\h)$ stands for the Fr\'echet space of
	$\h$-valued analytic functions on $U$. 
	A bounded linear operator $T$ is decomposable if and only if both $T$ and $T^*$ possess Bishop's property $(\beta)$.
	We refer to \cite{EschmeierPutinar} for details. 
	
	By combining the results above with the definition of $C$-symmetry, we obtain the following observation.
	
	\begin{proposition} 
		If $T$ is a bounded $C$-symmetric operator, then $T$ is decomposable if and only if $T$ satisfies Bishop's condition $(\beta)$.
	\end{proposition} 

	The articles \cite{JKL, JKLL,JKLL2} contain a host of related results concerning the local spectral
	theory of complex symmetric operators and we refer the reader there for further details and additional results.

%%%%%%%%%%%%%%%%%%%%%%%%%%%%%%%%%%
\section{Unbounded Complex Symmetric Operators}\label{SectionUnbounded}

\subsection{Basic definitions}

When extending Definition \ref{DefinitionCSO} to encompass unbounded operators, some care must be taken.
This is due to the fact that the term \emph{symmetric} means one thing when dealing with matrices and another when dealing with unbounded operators.

\begin{definition}
	Let $T : \dom(T) \to \h$ be a closed, densely defined linear operator acting on 
	$\h$ and let $C$ be a conjugation on $\h$.  We say that $T$ is \emph{$C$-symmetric}
	if $T \subseteq C T^* C$.
\end{definition}

Equivalently, the operator $T$ is $C$-symmetric if
\begin{equation}\label{EquationCTfg}
  \inner{ CTf,g } = \inner{ CTg,f}
\end{equation}
for all $f,g$ in $\mathcal{D}(T)$.   We say that an operator $T$
is $C$-\emph{selfadjoint} if $T = C T^\ast C$ (in particular, a bounded $C$-symmetric operator is $C$-selfadjoint).  
Unbounded $C$-selfadjoint operators are sometimes called \emph{$J$-selfadjoint},
although this should not be confused with the notion of $J$-selfadjointness in the theory of Kre\u{\i}n spaces
(in which $J$ is a \emph{linear} involution).

In contrast to the classical extension theory of von Neumann, 
it turns out that a $C$-symmetric operator always has a $C$-selfadjoint extension \cite{GlazmanBook,Glazman} (see also \cite{Galindo,Race}).
Indeed, the maximal \emph{conjugate-linear} symmetric operators $S$ (in the sense
that $\inner{Sf,g} = \inner{ Sg,f}$ for all $f,g$ in $\mathcal{D}(S)$) produce $C$-selfadjoint operators $CS$.
Because of this, we use the term \emph{complex symmetric operator} freely in both the bounded and unbounded
situations when we are not explicit about the conjugation $C$.  Much of this theory was developed
by Glazman  \cite{Glazman}.

In concrete applications, $C$ is typically derived from complex conjugation  
on an appropriate $L^2$ and $T$ is a non-selfadjoint differential operator.
For instance, the articles \cite{Knowles, Race} contain a careful analysis and parametrization of 
boundary conditions for Sturm-Liouville type operators with complex potentials which
define $C$-selfadjoint operators.  Such operators also arise in studies related to Dirac-type 
operators \cite{Cascaval}.  The complex scaling technique, a standard tool in the theory 
of Schr\"odinger operators, also leads to the consideration $C$-selfadjoint operators \cite{CSQ}
and the related class of \emph{$C$-unitary} operators \cite{Riss}.  

A useful criterion for $C$-selfadjointness can be deduced from the equality
\begin{equation*}
\mathcal{D}(C T^\ast C) = \mathcal{D}(T) \oplus \{ f \in \mathcal{D}(T^\ast C T^\ast C) : \ T^\ast C T^\ast C f + f = 0 \},
\end{equation*}
which is derived in \cite{Race}. 
A different criterion goes back to \v{Z}ihar$'$
\cite{Zihar}: if the $C$-symmetric
operator $T$ satisfies $\h = (T-zI)\mathcal{D}(T)$ for some complex
number $z$, then $T$ is $C$-selfadjoint. The resolvent set of $T$ consists of exactly the
points $z$ fulfilling the latter condition. We denote the inverse to the
right by $(T-zI)^{-1}$ and note that it is a bounded linear
operator defined on all of $\h$. We will return to these criteria in Subsection \ref{Subsection:RPD} below.
We focus now on the following important result.

\begin{theorem}
	If $T:\mathcal{D}(T)\to\h$ is a densely defined $C$-symmetric operator, then $T$ admits
	a $C$-selfadjoint extension.
\end{theorem}

	The history of this results dates back to von Neumann himself, who
	proved that every densely defined, $C$-symmetric operator $T$ which is also $C$-real,
	in the sense that $CT = TC$, admits a selfadjoint extension \cite{vonNeumann}.  Shortly thereafter,
	Stone demonstrated that an extension can be found that
	is $C$-real and hence $C$-selfadjoint \cite{Stone}.  Several decades passed before
	Glazman established that if $T$ is densely defined and \emph{dissipative} (meaning that
	$\Im \inner{Ax,x} \leq 0$ on $\mathcal{D}(T)$), then a dissipative $C$-selfadjoint 
	extension of $T$ exists \cite{Glazman}.
	
	Motivated by work on the renormalized field operators for the problem of the interaction of a 
	``meson'' field with a nucleon localized at a fixed point \cite{FriedrichsGalindo}, Galindo
	simultaneously generalized the von Neumann-Stone and Glazman results by 
	eliminating both the $C$-real and the dissipative requirements which had been placed upon $T$
	\cite{Galindo}.  Another proof was later discovered by Knowles \cite{Knowles}.

\begin{example}
	Consider an essentially bounded function $q : [-\pi,\pi] \rightarrow \C $ which
	satisfies $\Im q \geq 0$ and $\Re q \geq 1$ almost everywhere.  The operator
	\begin{equation*}
	[Tf] (x) = -f''(x) + q(x) f(x)
	\end{equation*}
	defined on the Sobolev space $W_0^2[-\pi,\pi]$ is dissipative and $C$-selfadjoint with respect to the canonical conjugation
	$Cf = \overline{f}$ \cite{CSO2}.
	By a deep Theorem of Keldysh \cite[Theorem V.10.1]{GK1},
	the eigenfunctions of $T$ are complete in $L^2[-\pi,\pi]$ and hence such operators are a prime candidates
	for analysis using the methods of Section \ref{Section:Spectral}.
\end{example}

\begin{example}
  Let $q(x)$ be a real valued, continuous, even function on $[-1,1]$
  and let $\alpha$ be a nonzero complex number satisfying $|\alpha| < 1$.
  For a small parameter $\epsilon>0$, we define the operator
  \begin{equation}\label{EquationTAE}
    [T_\alpha f](x) = -i f'(x) + \epsilon q(x) f(x),
  \end{equation}
  with domain
  \begin{equation*}
    \mathcal{D}(T_\alpha) = \{\, f \in L^2[-1,1]\,\,:\,\,  f' \in L^2[-1,1],\, \,f(1) = \alpha f(-1)\,\}.
  \end{equation*}
  Clearly $T_{\alpha}$ is a closed operator and $\dom(T_{\alpha})$ is dense in $L^2[-1,1]$.
  If $C$ denotes the conjugation $[Cu](x) = \overline{u(-x)}$
  on $L^2[-1,1]$, then the \emph{nonselfadjoint} operator $T_{\alpha}$ satisfies
  $T_{\alpha} = CT_{1/\overline{\alpha}}C$.  A short computation shows that
  $T_{\alpha}^* = T_{1/\overline{\alpha}}$ and hence $T_{\alpha}$ is $C$-selfadjoint.
\end{example}

\begin{example}
  Consider a Schr\"odinger operator $H:{\mathcal D}(\nabla^2)\rightarrow L^2(\R^d)$
  defined by $H = -\nabla^2+v(\vec{x} )$
  where the potential $v(\vec{x} )$ is dilation analytic in a finite strip $|\Im\theta|<I_0$ and $\nabla^2$-relatively compact. 
  The standard dilation 
  \begin{equation*}
    [U_{\theta}\psi](\vec{x} ) = e^{d \theta/2}\psi(e^\theta\vec{x} )
  \end{equation*}
  allows us to define an analytic (type $A$) family of operators:
  \begin{equation*}
    H_\theta \equiv U_{\theta}HU_{\theta}^{-1}=-e^{-2\theta} \nabla^2 + v(e^\theta \vec{x} ),
  \end{equation*}
  where $\theta$ runs in the finite strip $|\Im\theta|<I_0$ (see \cite{RS4} for definitions).
  It is readily verified that the scaled Hamiltonians $H_\theta$ are $C$-selfadjoint with respect to
  complex conjugation $Cf=\overline{f}$.  
\end{example}

\subsection{Refined polar decomposition}\label{Subsection:RPD}

If an unbounded $C$-selfadjoint
operator has a compact resolvent, then a canonically associated
antilinear eigenvalue problem always has a complete set of mutually
orthogonal eigenfunctions \cite{CSQ, CSO2}:

\begin{theorem}\label{RPDUB}
 If $T: \mathcal{D}(T) \to H$ is an
 unbounded $C$-selfadjoint operator with compact resolvent
 $(T-zI)^{-1}$ for some complex number $z$, then there exists an
 orthonormal basis $\{\vec{u}_n\}_{n=1}^\infty$ of $\h$ consisting of
 solutions of the antilinear eigenvalue problem:
 $$
   (T-zI)\vec{u}_n = \lambda_n C\vec{u}_n
$$
 where $\{\lambda_n\}_{n=1}^{\infty}$ is an increasing sequence
 of positive numbers tending to $\infty$.
\end{theorem}

This result is a consequence of the refined polar
decomposition for \emph{bounded} $C$-symmetric operators
described in Theorem~\ref{RPDB}. The preceding result provides a useful
tool for estimating the norms of resolvents of certain unbounded operators.

\begin{corollary}
  If $T$ is a densely-defined $C$-selfadjoint operator
  with compact resolvent
  $(T-zI)^{-1}$ for some complex number $z$, then
  \begin{equation}\label{ResolventEstimate}
    \norm{(T - zI)^{-1}} = \frac{1}{\inf_n \lambda_n}
  \end{equation}
  where the $\lambda_n$ are the positive solutions to the
  antilinear eigenvalue problem:
 \begin{equation}\label{Anti}
   (T-zI)\vec{u}_n = \lambda_n C\vec{u}_n.
 \end{equation}
\end{corollary}

We also remark that the refined polar decomposition $T= CJ|T|$
applies, under certain circumstances, to unbounded $C$-selfadjoint
operators:

\begin{theorem}
  If $T: \mathcal{D}(T) \to \h$ is a densely
  defined $C$-selfadjoint operator with zero in its resolvent,
  then $T = CJ|T|$ where $|T|$ is a positive selfadjoint operator
  (in the von Neumann sense) satisfying
  $\mathcal{D}(|T|) = \mathcal{D}(T)$ and $J$ is a conjugation
  on $\h$ that commutes with the spectral measure
  of $|T|$. Conversely, any operator of the form described above is
  $C$-selfadjoint.
\end{theorem}

\subsection{$C$-selfadjoint extensions of $C$-symmetric operators}

The theory of $C$-selfadjoint extensions of $C$-symmetric operators is parallel to von Neumann's
theory of selfadjoint extensions of a symmetric operator. It was the Soviet school that developed the
former, in complete analogy, but with some unexpected twists, to the later. Two early contributions are
\cite{Visik,Zihar} complemented by Glazman's lucid account \cite{Glazman}.

A convenient $C$-selfadjointness criterion is offered by the following observation of 
\v{Z}ihar$'$ \cite{Zihar}.

\begin{theorem}
If $T$ is a $C$-symmetric operator such that
$\ran (T-\lambda)\dom(T) = \h$ for some complex number $\lambda$, then $T$ is $C$-selfadjoint.
\end{theorem}

One step further, to have an effective description of all $C$-selfadjoint extensions of an operator $T$
one assumes (after Vi\v{s}ik \cite{Visik}) that there exists a point $\lambda_0 \in \C$ and a positive constant
$\gamma$ with the property
$$\qquad\qquad \| (T-\lambda_0 I)\vec{x} \| \geq \gamma \| \vec{x} \|, \qquad \vec{x} \in \dom(T).$$
Then one knows from \v{Z}ihar$'$ \cite{Zihar} that there are $C$-selfadjoint extensions which are also
bounded from below at $\lambda_0$.  Consequently, the familiar von Neumann parametrization
of all such extensions $\widetilde{T}$ in terms of a direct sum decomposition is available:
$$ \dom(C T^\ast C) = \dom(T) + (\widetilde{T}-\lambda_0 I)^{-1} \ker (T^\ast - \overline{\lambda_0} I)
+ C \ker (T^\ast - \overline{\lambda_0}I).$$ Consequently $\dim \ker (T^\ast - \overline{\lambda_0} I)$
is constant among all points $\lambda$ for which $(T-\lambda  I)$ is bounded from below.

The analysis of $C$-selfadjoint extensions is pushed along the above lines by Knowles \cite{Knowles},
who provided efficient criteria applicable, for instance, to Sturm-Liouville operators of any order. We reproduce below an
illustrative case.

\begin{example}
Let $[a,\infty)$ be a semi-bounded interval of the real line and let $p_0, p_1$ denote Lebesgue
integrable, complex valued functions on $[a,\infty)$ such that $p_0'$ and $1/p_0$ are also integrable. We define the
Sturm-Liouville operator 
$$ \tau(f) = - ( p_0 f')' + p_1 f$$
with maximal domain, in the sense of distributions, $\dom(T_{\max}) \subseteq L^2[a,\infty)$.
By choosing $C$ to be complex conjugation we remark that $\tau$ is formally 
$C$-symmetric.
One can define the minimal closed operator $T_{\min}$ having as graph
the closure of  $(f,\tau(f))$ with $f \in \dom(T_{\max})$ of compact support in $(a,b)$. 
Then
$CT_{\min}^\ast C = T_{\max}$, hence
$T_{\min}$ is $C$-symmetric 
and there are regularity points in the resolvent of $T_{\min}$.
Assume that the deficiency index is equal to one, that is $\dim {\rm Ker}(T_{\max} - \lambda_0) = 1$
for some point $\lambda_0 \in \C$.
Any regular $C$-selfadjoint extension $\widetilde{T}$ of $T_{\min}$ is the restriction of $\tau$
to a domain
$$ \dom(T_{\min}) \subseteq \dom(\widetilde{T}) \subseteq \dom(T_{\max}) $$ 
specifically described by a pair of complex numbers $(\alpha_0, \alpha_1)$:
$$ \dom(\widetilde{T}) = \{ f \in \dom(T_{\max});\ \ \alpha_0 f(a) + \alpha_1 p_0(a) f'(a) = 0 \}.$$
\end{example}

The existence of regular points in the resolvent set of a $C$-symmetric operator is not guaranteed.
However, there are criteria that guarantee this; see \cite{Knowles,Race}.
The anomaly in the following example is resolved in an ingenious way by Race \cite{Race}
by generalizing the notion of resolvent.

\begin{example} We reproduce from 
\cite{McLeod}
an example of simple Sturm-Liouville operator without regular points in the resolvent.
Consider on $[0,\infty)$ the operator
$$ \tau(f)(x) = - f''(x) -2i e^{2(1+i)x} f(x).$$
Then for every $\lambda \in \C$ there are no solutions $f$ of $\tau f = \lambda f$
belonging to $L^2[0,\infty)$.
\end{example}

Finally, we reproduce a simple but illustrative example considered by Krej\v{c}i\v{r}\'ic, Bila and Znojil \cite{KBZ}.

\begin{example} Fix a positive real number $d$. Let $H_\alpha f = - f''$ defined on the Sobolev space $W^{2,2}([0,d])$ with boundary conditions
$$ f'(0) +i\alpha f(0) = 0, \ \ f'(d)+i\alpha f(d) = 0,$$
where $\alpha$ is a real parameter.
Then the operator $H_\alpha$ is $C$-symmetric, with respect to the standard $\mathcal{PT}$-symmetry
$[Cf] (x) = \overline{f(d-x)}$, that is $H^\ast_\alpha = H_{-\alpha}$.

It turns out by simple computations that the spectrum of $H_\alpha$ is discrete, with only
simple eigenvalues if $\alpha$ is not an integer multiple of $\pi/d$:
$$ \sigma(H_\alpha) = \Big\{ \alpha^2, \frac{\pi^2}{d^2}, \frac{2^2 \pi^2}{d^2}, \frac{3^2 \pi^2}{d^2}, \ldots \Big\}.$$
The eigenfunctions of $H^\ast_\alpha$ are computable in closed form:
$$ h_0(x) = \sqrt{1/d} + \frac{e^{i\alpha x}-1}{\sqrt{d}}$$
corresponding to the eigenvalue $\alpha^2$, and respectively
$$ h_j(x) = \sqrt{2/d}\Big[ \cos \Big(\frac{j \pi x}{d}\Big) + i \frac{d \alpha}{j \pi} \sin \Big(\frac{j \pi x}{d}\Big) \Big]$$
corresponding to the eigenvalues $\frac{j^2 \pi^2}{d^2}, \ \ j \geq 1.$

Remarkably, these eigenfunctions form a Riesz basis in $L^2([0,d])$, whence the operator $H_\alpha$
can be "symmetrized" and put in diagonal form in a different Hilbert space metric which turns
the functions $h_k, \ k \geq 0,$ into an orthonormal base. See also \cite{KrejcirikJPA2008vu,KSZ}. { One should be aware that this is not a general rule, as there are known examples, such as the even non-selfadjoint anharmonic oscillators, where the eigenfunctions form a complete set but they do not form a basis in the Hilbert, Riesz or Schauder sense \cite{HenryJST2013cu}. The cubic harmonic oscillator, to be discussed below, is also an example displaying same phenomenon.}

\end{example}

\section{$\mathcal{P T}$-symmetric Hamiltonians}

The question of what is the correct way to represent an observable in quantum mechanics has been brought up more often lately. Among its axioms, the traditional quantum theory says that the classical observables are represented by selfadjoint operators whose spectrum of eigenvalues represents the set of values one can observe during a physical measurement of this observable. It has been noted, however, that the selfadjointness of an operator, which can be seen as a symmetry property relative to complex conjugation and transposition, can be replaced with other types of symmetries and the operator will still posses a set of real eigenvalues. 

A good introduction to the subject is the paper by Bender \cite{BenderRPP2007vg} where the reader can also find a valuable list of references.  A personal view of the role of non-hermitian operators in quantum mechanics is contained in Znojil's article \cite{Znojil2}. The aficionados of $\mathcal{PT}$-symmetry
in quantum physics maintain an entertaining and highly informative blog 
\url{http://ptsymmetry.net/}, while a serious criticism was voiced by Streater \url{http://www.mth.kcl.ac.uk/~streater/lostcauses.html\#XIII}.
We seek here only to comment on the connection between $\mathcal{PT}$-symmetric and complex symmetric operators.

\subsection{{ Selected Results} }

The work by Bender and Mannheim \cite{BenderPLA2010hg} resulted in a set of necessary and sufficient conditions for the reality of energy eigenvalues of finite dimensional Hamiltonians. The first interesting conclusion of this work is the fact that for the 
secular equation
$$\det(H-\lambda I)=0$$
to contain only real coefficients,
the Hamiltonian must necessarily obey
$$(\mathcal{P}\mathcal{T}) H (\mathcal{P}\mathcal{T})^{-1} = H,$$
where $\mathcal{P}$ is a unitary matrix with $\mathcal{P}^2=1$ and $\mathcal{T}$ is a conjugation. In many examples of interest, one can identify 
$\mathcal{P}$ with the parity operator and $\mathcal{T}$ with the time-reversal operator (this excludes fermionic systems for which $\mathcal{T}^2=-1$). Hence, the reality of the energy eigenvalues always requires some type of $\mathcal{P}\mathcal{T}$ symmetry, but this condition alone is generally not sufficient.

For diagonalizable finite dimensional $\mathcal P \mathcal T$-symmetric Hamiltonians, the following criterion gives a sufficient condition. Consider the set $\mathfrak C$ of operators $\mathcal C$ that commute with $H$ and satisfy $\mathcal C^2 = 1$. Note that if $P$ is the spectral projection for an eigenvalue, then 
$\mathcal{C} = P - P^{\perp}$ satisfies these conditions. 
The criterion for the reality of the spectrum says that if every $\mathcal C$ from $\mathfrak C$ commutes with $\mathcal P \mathcal T$, then all 
of the eigenvalues of $H$ are real. If at least one such $\mathcal C$ does not commute with $\mathcal P \mathcal T$, then the spectrum of $H$ 
contains at least one conjugate pair of complex eigenvalues.

For non-diagonalizable Hamiltonians, Bender and Mannheim  derived the following criterion: the eigenvalues of any nondiagonalizable Jordan 
block matrix that possesses just one eigenvector will all be real if the block is $\mathcal{PT}$-symmetric, and will all be complex if the block is not $\mathcal{PT}$-symmetric.

The reality of the energy spectrum of a $\mathcal P \mathcal T$-symmetric Hamiltonian is only part of the story because to build a quantum theory with a probabilistic interpretation one needs a unitary dynamics. One useful observation in this direction is that a non-Hermitian $\mathcal P \mathcal T$-symmetric Hamiltonian becomes Hermitian with respect to the inner product
\begin{equation*}
(f,g)_{\mathcal P \mathcal T} =  \inner{ \mathcal P \mathcal T \mathcal K f, g},
\end{equation*}
where $\mathcal{K}$ denotes ordinary complex conjugation.
The shortcoming of the construction is that ${(\,\cdot\,,\,\cdot\,)}_{\mathcal P \mathcal T}$ is indefinite. 
The hope is then in finding an additional complex linear symmetry
$\mathcal{C}$ which commutes with the hamiltonian, so that inner product
$$(f,g)_{\mathcal{CPT}} = \inner{\mathcal{PTCK}f,g}$$
is positive definite.
The work \cite{BenderJPA2012fr} highlighted some interesting possibilities in this respect. Specifically, it was shown that if the symmetry transformation $\mathcal{C}$ is bounded, then indeed the $\mathcal P \mathcal T$-symmetric Hamiltonian can be realized as Hermitian operator on the same functional-space but endowed with a new scalar product. In contradistinction, if the symmetry transformation $\mathcal C$ is unbounded, then the original  $\mathcal P \mathcal T$-symmetric operator has selfadjoint extensions but in general is not essentially selfadjoint. That means, it accepts more than one selfadjoint extension, and the possible extensions describe distinct physical realities. The extensions are defined in a functional-space that is strictly larger than the original Hilbert space.

In the same direction, a cluster of recent discoveries \cite{AK1, AK2, AK3, AGK} provided rigorous constructions of the symmetries $\mathfrak C$ above from additional hidden symmetries of the original operator. In particular,
motivated by carefully chosen examples, Albeverio and Kuzhel combine in a novel and ingenious manner von Neumann's classical theory
of extensions of symmetric operators, spectral analysis in a space with an indefinite metric, and elements of Clifford algebra. Notable is their adaptation of scattering theory to the study of $\mathcal P \mathcal T$-selfadjoint extensions
of $\mathcal P \mathcal T$-symmetric operators. We refer to \cite{AK3} for details, as the rather complex framework necessary to state the main results contained in that paper cannot be reproduced in our survey.

\begin{example} The perturbed {\it cubic oscillator operator} 
$$ T_\alpha y = -y'' + i x^3 y + i\alpha x y, \ \  \alpha \geq 0,$$
defined with maximal domain on $L^2(\R, dx)$ served as a paradigm during the evolution period of $\mathcal{PT}$ quantum mechanics.
It is a complex symmetric operator, $T_\alpha^\ast  = CT_\alpha C$, with respect to the $\mathcal{PT}$-conjugation
$$ Cf (x) = \overline{f(-x)}.$$ The reality of its spectrum was conjectured in 1992 by Bessis and Zinn-Justin.
The conjecture was numerically supported by the work of Bender and Boettcher \cite{MR1627442} and settled into the affirmative
by Shin \cite{Shin2002} and Dorey-Dunning-Tateo \cite{DDT2007}. { The rigorous analysis of the last two references rely on classical PDE techniques such as  asymptotic analysis in the complex domain, WKB expansions, Stokes lines, etc.} The survey by Giordanelli and Graf \cite{GiordanelliARXIV2013hg} offers a sharp, lucid account of these asymptotic expansions. The next section will be devoted to a totally different method of proving the reality of the spectrum of the operator $T_\alpha$, derived this time from perturbation theory in Krein space.

In a recent preprint Henry \cite{HenryARXIV2013re} concludes that the operator $T_\alpha$ is not similar to a selfadjoint operator 
by estimating the norm of the spectral projection on the $n$th eigenvalue, and deriving in particular that the eigenfunctions of $T_\alpha$ do not form a Riesz basis, a result already proved by Krej\v{c}i\v{r}\'ik and Siegl \cite{KS1}.

\end{example}

We select in the subsequent sections a couple of relevant and mathematically complete results pertaining to the flourishing topics of non-Hermitian quantum physics.

\subsection{Perturbation theory in Kre\u{\i}n space}

Among the rigorous explanations of the reality of the spectrum of a non-selfadjoint 
operator, perturbation arguments play a leading role. In particular, perturbation theory in Kre\u{\i}n space was succesfully used by Langer and Tretter
\cite{LangerTretter,LangerTretterCorrigendum}. The thesis of Nesemann \cite{NesemannPT} contains 

A \emph{Kre\u{\i}n space} is a vector space $\K$ endowed with an inner product $\{\cdot, \cdot\}$, such that
there exists a direct sum orthogonal decomposition
$$ \K = \h_{+} + \h_{-},   \ \ \ \{ \h_+, \h_-\} = 0,$$
in which $(\h_{+}, \{\cdot, \cdot\}), (\h_{-}, -\{\cdot, \cdot\})$ are Hilbert spaces. 
Note that such a decomposition is not unique, as a two dimensional indefinite example immediately shows. The underlying positive definite form
$$\langle \cdot, \cdot \rangle = \{\cdot, \cdot\}|_{\h_+} - \{\cdot, \cdot\}|_{\h_-}$$
defines a Hilbert space structure on $\K$. In short, a Kre\u{\i}n space corresponds to a linear, unitary involution
$J$, acting on a Hilbert space $\K$, with the associated product
$$ \{ \vec{x}, \vec{y} \} = \langle J\vec{x}, \vec{y} \rangle.$$

For a closed, densely defined operator $T$ on $\K$, the Kre\u{\i}n space adjoint $T^{[\ast]}$ satisfies
$$\qquad\qquad \{T \vec{x}, \vec{y}\} = \{\vec{x}, T^{[\ast]} \vec{y}\}, \qquad \vec{x} \in \dom(T),  \vec{y} \in \dom(T^{[\ast]}).$$ 
The operator $T$ is selfadjoint (sometimes called $J$-selfadjoint)
if $T = T^{[\ast]}$, that is $T^\ast J = J T$.  
A subspace $\E \subseteq \K$ is called positive if $\{ \vec{x}, \vec{x} \} \geq 0$ for all $\vec{x} \in \E$ and {\it uniformly positive}
if there exists a constant $\gamma>0$ such that
$$\qquad\qquad \{ \vec{x},\vec{x}\} \geq \gamma \| \vec{x} \|^2, \qquad \vec{x} \in \E.$$

In the most important examples that arise in practice, a second hidden symmetry is present in the structure of a
$J$-selfadjoint operator (in the sense of Kre\u{\i}n spaces), bringing into focus the main theme of our survey. Specifically, assume that
there exists a conjugation $C$, acting on the same Hilbert space as the linear operators $T$ and $J$, satisfying
the commutation relations:
$$ CT = TC, \ \ \ \ CJ = JC.$$ Then  the $J$-selfadjoint operator $T$ is also $CJ$-symmetric:
$$ T^\ast CJ = T^\ast JC = J T C = JC T = CJ T.$$

Operator theory in Kre\u{\i}n spaces is well-developed, with important applications to continuum mechanics and function theory;
see the monograph \cite{Azizov}.
An important result of Langer and Tretter states that a continuous family of selfadjoint (unbounded)
operators in a Kre\u{\i}n space preserves the uniform positivity of spectral subspaces obtained by Riesz projection
along a fixed closed Jordan curve. The details in the statement and the proof are contained in the two
notes \cite{LangerTretter,LangerTretterCorrigendum}. We confine ourselves to reproduce a relevant example for our survey.

\begin{example} Let $\K = L^2([-1,1], dx)$ be the Kre\u{\i}n space endowed with the inner product
$$ \{f,g\} = \int_{-1}^1 f(x) \overline{g(-x)} dx.$$ The positive space $\h_+$ can be chosen to consist of
all even functions in $\K$, while the negative space to be formed by all odd functions.

Let $V \in L^\infty[-1,1]$ be a $\mathcal{PT}$-symmetric function,
that is
$$ V(-x) = \overline{V(x)}.$$ Then the Sturm-Liouville operator
$$ Tf (x) = - f ''(x) + V(x) f(x),$$
with domain $\dom(T) = \{ f \in \K; \ f(-1) = f(1) = 0\}$ is symmetric in Kre\u{\i}n space sense.
More precisely, let 
$$ (Jf)(x) = f(-x), \ \ \ f \in L^2[-1,1]$$
be the unitary involution (parity) that defines the Kre\u{\i}n space structure and let
$C$ denote complex conjugation: $(Cf) = \overline{f}$. Note that $CJ=JC$. The $J$-symmetry of the 
operator $T$ amounts to the obvious identity (of unbounded operators):
$$ T^\ast J = JT.$$ On the other hand
$$ T^\ast C = C T,$$
and
$$ T CJ = CJ T.$$ Therefore we are dealing with a $C$-symmetric operator $T$ commuting with the conjugation
$$(CJf) (x) = \mathcal{PT} f (x) = \overline{f(-x)}.$$

By means of the linear deformation $T_\epsilon f (x) = - f ''(x) + \epsilon V(x) f(x),
\ 0 \leq \epsilon \leq 1$, the conclusion of \cite{LangerTretter} is that, assuming
$$ \| V \|_{\infty} < \frac{3 \pi^2}{8},$$
one finds that the spectrum of $T$ consists of simple eigenvalues $\lambda_j$, all real, alternating
between positive and negative type, and satisfying
$$ \left|\lambda_j - \frac{j^2 \pi^2}{4}\right| \leq \| V \|_{\infty}.$$
In particular one can choose $V(x) = i x^{2n+1}$ with an integer $n \geq 0$.
\end{example}

\subsection{Similarity of differential $C$-symmetric operators}

The intriguing question why certain $\mathcal{PT}$-symmetric hamiltonians with complex potential have real
spectrum is still open, in spite of an array of partial answers and a rich pool of examples, see
\cite{MR1627442,MR1686605,Most1, Most2, Most3}.

 The recent works
\cite{CGS, CGHS}
offer a rigorous mathematical explanation for the reality of the spectrum for a natural class of Hamiltonians.
We reproduce below a few notations from this article and the main result.

The authors are studying an algebraic, very weak form of similarity between two closed, densely defined
linear operators $A_j : \dom(A_j) \to \h, \ j=1,2$.  Start with the assumption that
both spectra $\sigma(A_1), \sigma(A_2) \subseteq \C$ are discrete and consist of eigenvalues of
finite algebraic multiplicity. That is, for a point $\lambda \in \sigma(A_j)$ there exists  a finite dimensional
space (of generalized eigenvectors) $E^{(j)}(\lambda) \subseteq \dom(A_j)$ satisfying
$$ E^{(j)}(\lambda)  = \ker (A_j-\lambda I)^N,$$
for $N$ large enough. Assume also that there are linear subspaces $V_j \subseteq \dom(A_j), \ j =1,2,$
such that
$$ \bigcup_{\lambda \in \sigma(A_j)} E^{(j)}(\lambda) \subseteq V_j$$ and
$$ A_j V_j \subseteq V_j, \ \ j=1,2.$$

The operators $A_j$ are called {\it similar} if there exists an invertible linear transformation
$X: V_1 \to V_2$ with the property $X A_1 = A_2 X.$ Then it is easy to prove that
$\sigma(A_1) = \sigma(A_2)$.
If, under the above similarity condition, the operator $A_1$ is selfadjoint, then the spectrum of
$A_2$ is real. This general scheme is applied in \cite{CGS} to a class of differential operators as follows.

Let $q(x,\xi)$ be a complex valued quadratic form on $\R^d \times \R^d$ so that $\Re q$ is positive definite.
The $\mathcal{PT}$-symmetry of the operator with symbol $q$ is derived from an abstract $\R$-linear involution
$\kappa : \R^d \to \R^d$, so that
$$\qquad\qquad q(x,\xi) = \overline{ q(\kappa(x), - \kappa^t (\xi))}, \qquad (x,\xi) \in \R^d \times \R^d.$$ Let $Q$
denote Weyl's quantization of the symbol $q$, that is the differential operator
$$ Q = \sum_{|\alpha + \beta| = 2} q_{\alpha, \beta} \frac{x^\alpha D^\beta + x^\beta D^\alpha}{2},$$
where $D$ stands as usual for the tuple of normalized first order derivatives $D_k = -i \frac{\partial}{\partial \vec{x}_k}.$ It is known that the maximal closed realization of $Q$ on the domain
$$ \dom(Q) = \{ u \in L^2(\R^d); \ Qu \in L^2(\R^d) \}    $$
coincides with the graph closure of the restriction of $Q$ to the Schwarz space ${\mathcal S}(\R^d).$
The operator $Q$ is elliptic, with discrete spectrum and $\mathcal{PT}$-symmetric, that is
$ [ Q, \mathcal{PT}] =0$, where $\mathcal{PT}(\phi)(x) = \overline{\phi(\kappa(x))}.$ Attached to the symbol $q$
there is the fundamental matrix $F : \C^{2d} \to \C^{2d}$, defined by
$$\qquad\qquad q(X,Y) = \sigma(X,FY), \qquad  X,Y \in \C^{2d},$$
where $q(X,Y)$ denotes the polarization of $q$, viewed as a symmetric bilinear form
on $\C^{2d}$ and $\sigma$ is the canonical complex symplectic form on $\C^{2d}$.

Under the above conditions, a major result of Caliceti, Graffi, Hitrik, Sj\"ostrand \cite{CGHS} is the following.

\begin{theorem}Assume that $\sigma(Q) \subseteq \R$. Then
the operator $Q$ is similar, in the above algebraic sense, to a selfadjoint operator if and only
if the matrix $F$ has no Jordan blocks.
\end{theorem}

The reader can easily construct examples based on the above criterion. The same article \cite{CGHS} contains an analysis of the following example.

\begin{example}
Let $$Q = -\Delta + \omega_1 \vec{x}_1^2 + \omega_2^2 \vec{x}_2^2 + 2i g \vec{x}_1 \vec{x}_2,$$
where $\omega_j >0, \ j=1,2, \ \omega_1 \neq \omega_2$ and $g \in \R$. The operator $Q$ is
globally elliptic and $\mathcal{PT}$-symmetric, with respect to the involution $\kappa(\vec{x}_1,\vec{x}_2) = (-\vec{x}_1,\vec{x}_2)$.
This operator appears also in a physical context \cite{Cannata}.

The above theorem shows that the spectrum of $Q$ is real precisely when
$$ -|\omega_1^2 - \omega_2^2| \leq 2g \leq |\omega_1^2 - \omega_2^2| $$
while $Q$ is similar to a selfadjoint operator if and only if
$$ -|\omega_1^2 - \omega_2^2| < 2g < |\omega_1^2 - \omega_2^2|. $$
\end{example}

\subsection{Pauli equation with complex boundary conditions}

An interesting example of a $\mathcal P \mathcal T$-symmetric spin-$\frac{1}{2}$ system is the Pauli Hamiltonian \cite{KochanJPA2012re}:
$$
H= - {\bm \nabla}^2 + {\bm B}\cdot {\bm L} + ({\bm B} \times \vec{x} )^2 + {\bm B}\cdot{\bm \sigma}
$$
defined on the Hilbert space $L^2(\Omega \in \R^2)\otimes \C ^2$. The domain of $H$ is defined by boundary condition:
$$
\frac{\partial \psi}{\partial {\bm n}} +A \psi =0, \ \mbox{on} \ \partial \Omega,
$$
where ${\bm n}$ is the outward pointing normal to the boundary and $A$ is a $2 \times 2$ complex-valued matrix. Above, ${\bm B}$ represents a magnetic field and all the physical constants were set to one.

The selfadjoint property of the Hamiltonian can be broken to a $\mathcal P \mathcal T$-symmetry by such boundary conditions. Interestingly, the same type of boundary condition, when numerically tuned, can lead to situations where the eigenvalue spectrum is entirely real or entirely complex. This example is also interesting because the time reversal
transformation is given by:
\begin{equation}
\mathcal T \left(
\begin{array}{c}
\psi_+(\vec{x} ) \\
\psi_-(\vec{x} )
\end{array}
\right )
=i \left(
\begin{array}{c}
\overline{\psi_-(\vec{x} )} \\
-\overline{\psi_+(\vec{x} )}
\end{array}
\right ), \ \mathcal T^2 = -1,
\end{equation}
as appropriate for spin-$\frac{1}{2}$ systems. The parity operation acts as usual $\mathcal P \psi(\vec{x} ) = \psi(-\vec{x} )$. 
  
Reference~\cite{KochanJPA2012re} analyzed the model in some simplifying circumstances, namely, for ${\bm B}=(0,0,B)$ in which case ${\bm B}\cdot {\bm L}$ and $({\bm B} \times \vec{x} )^2$ act only on the first two coordinates and ${\bm B}\cdot {\bm \sigma}$ reduces to $B\sigma_3$. The domain was taken to be $\Omega = \R^2 \times (-a,a)$ and the matrix $A$ entering the boundary condition was taken independent of the first two space-coordinates. 
Under these conditions, the model separates into a direct sum of two terms, out of which the term acting on the third space-coordinate $x$ 
is of interest to us
$$
H_b = 
\begin{bmatrix}
-\frac{d^2}{dx^2}+b & 0 \\
0 & - \frac{d^2}{dx^2} - b 
\end{bmatrix},
$$
which is defined on the Hilbert space $\mathcal H = L^2((-a,a),\C ^2)$ and subjected to the boundary conditions:
\begin{equation}
\frac{d\psi}{dx}(\pm a) + A^\pm \psi(\pm a)=0.
\end{equation}
The boundary conditions preserving the $\mathcal P \mathcal T$-symmetry of the system are those with:
$$A^- = \mathcal T A^+ \mathcal T.
$$

The analysis of the spectrum led to the following conclusions. 
\begin{enumerate}\addtolength{\itemsep}{0.5\baselineskip}
\item The residual spectrum is absent.
\item $H_b$ has only discrete spectrum.
\item In the particular $\mathcal P \mathcal T$-symmetric case:
$$
A^{\pm} = 
\begin{bmatrix}
i \alpha \pm \beta & 0 \\
0 & i \alpha \pm \beta
\end{bmatrix},
$$
with $\alpha$, $\beta$ real parameters, and $\beta \geq 0$, the spectrum of $H_b$ is always entirely real. If $\beta <0$, then complex eigenvalues may show up in the spectrum.
\end{enumerate}

\section{Miscellaneous applications}

We collect below a series of recent applications of complex symmetric operators to a variety of mathematical and
physical problems.

\subsection{Exponential decay of the resolvent for gapped systems}
This is an application taken from \cite{CSQ}.  Let $-{\bm \nabla}_D^2$ denote the Laplace operator with Dirichlet
boundary conditions over a finite domain $\Omega \subseteq \R^d$ with smooth boundary. Let $v(\vec{x} )$ be a scalar potential, which is ${\bm \nabla}_D^2$-relatively bounded with
relative bound less than one, and let $\vec{A}(\vec{x})$ be a smooth magnetic vector potential. The following Hamiltonian:
$$H_\vec{A} :{\mathcal D}({\bm \nabla}_D^2) \to L^2(\Omega),\ \
H_\vec{A} =-({\bm \nabla}+i\vec{A})^2+v(\vec{x} ),$$ 
generates the quantum dynamics of electrons in a material subjected to a magnetic field ${\bm B}=\nabla\times\vec{A} $. We will assume that this material is an insulator and that the magnetic field is weak. In this regime, even with the boundary, the energy spectrum of $H_\vec{A} $ will generically display a spectral gap $[E_-,E_+] \subseteq \rho(H_\vec{A} )$. This will be one of our assumptions. There is a great interest in sharp exponential decay estimates on the resolvent $(H_\vec{A} -E)^{-1}$ with $E$ in the spectral gap \cite{Prodan:2005qi}.

In the theory of Schr\"odinger operators, non-selfadjoint operators are often generated by conjugation with non-unitary transformations, such as: 

\begin{definition} 
Given an arbitrary $\vec{q} \in \R^d$ ($q\equiv|\vec{q} |$), let $U_\vec{q} $ denote the following
bounded and invertible map
$$U_\vec{q} :L^2(\Omega)\rightarrow L^2(\Omega), \ \ [U_{\vec{q}}f](\vec{x} )=e^{\vec{q} \vec{x} }f(\vec{x} ),$$ which leaves the
domain of $H_\vec{A} $ unchanged. 
\end{definition}

The conjugation of $H_\vec{A} $ with the transformation $U_\vec{q} $ defines a family of (non-selfadjoint) scaled Hamiltonians:
 $$H_{\vec{q} ,\vec{A} }\equiv U_\vec{q} H_\vec{A}  U_{\bm
q}^{-1}, \ \ \vec{q} \in \R^d.$$ 
The scaled Hamiltonians are explicitly given by
\begin{equation}
    H_{\vec{q} ,\vec{A} }:{\mathcal D}({\bm \nabla}_D^2)\rightarrow L^2(\Omega),
    \ \ H_{\vec{q} ,\vec{A} }=H_\vec{A} +2\vec{q} ({\bm \nabla}+i\vec{A} )-q^2.
\end{equation}
Note that $H_{\vec{q} ,\vec{A} }$ are not  $C$-symmetric operators, with respect to any natural
conjugation. The following construction fixes this shortcoming. 

\begin{lemma} Consider the following block-matrix operator
$\mathbf{H}$ and the conjugation $C$ on $L^2(\Omega)\oplus L^2(\Omega)$:
$$
    \mathbf{H}=
    \begin{bmatrix}
      H_{\vec{q} ,\vec{A} } & 0 \\
      0 & H_{-\vec{q} ,-\vec{A} } \\
    \end{bmatrix},\qquad
    C   = 
    \begin{bmatrix}
      0 & {\mathcal C} \\
      {\mathcal C} & 0 \\
    \end{bmatrix},
$$
where $\mathcal{C}$ is the ordinary complex conjugation. Then $\mathbf{H}$ is $C$-selfadjoint:
$\mathbf{H}^\ast=C\mathbf{H}C$. Moreover,
\begin{equation}\label{NormEq}
    \|(\mathbf{H}-E)^{-1}\|=\|(H_{\vec{q} ,\vec{A} }-E)^{-1}\|=\|(H_{-{\bm    q},-\vec{A} }-E)^{-1}\|.
\end{equation}
\end{lemma} 

\begin{proof}
The statement follows from $H^\ast_{\vec{q} ,\vec{A} }=H_{-\vec{q} ,\vec{A} }$ and ${\mathcal C}H_{\vec{q} ,\vec{A} }=H_{\vec{q} ,-\vec{A} }{\mathcal C}$.
\end{proof}

The refined polar decomposition for $C$-selfadjoint operators and its consequences permit
sharp estimates on the resolvent of the scaled Hamiltonians. Indeed, according to Theorem 
\ref{RPDUB}, the antilinear eigenvalue problem (with $\lambda_n\geq0$)
\begin{equation}\label{EigP}
    (\mathbf{H}-E)\phi_n=\lambda_n C\phi_n
\end{equation}
generates an orthonormal basis $\phi_n$ in $L^2(\Omega)\oplus L^2(\Omega)$
and
\begin{equation}\label{bound}
    \|(\mathbf{H}-E)^{-1}\|=\frac{1}{\min_n \lambda_n}.
\end{equation}
The task is then to generate a lower bound on the sequence $\{\lambda_n\}$. 
The advantage of using the antilinear eigenvalue equations is that one can find explicit (but somewhat formal) expressions for the $\lambda$'s. Indeed, if one writes $\phi_n=f_n\oplus \vec{g}_n$, then:
\begin{equation}\label{LambdaEq}
    \lambda_n=\frac{|\langle f_n,|H_\vec{A} -E-q^2|f_n\rangle
    +4 \Re\langle f_n,P_+[\vec{q} ({\bm \nabla}+i\vec{A} )]P_-f_n\rangle | }{|\Re\langle
    Sf_n,\overline{g}_n \rangle|},
\end{equation}
where $S=P_+-P_-$ and $P_\pm$ are the spectral projections of $H_\vec{A} $ for the upper/lower (relative to the gap) part of the spectrum. These formal expressions have already separated a large term (the first term in the 
denominator), which can be controlled via the spectral theorem for the selfadjoint operator $H_\vec{A} $, and a small term (the second term in the nominator), which can be estimated approximately. The following lower bound emerges.  

\begin{proposition} \label{Proposition:EWTT}
$$ \lambda_n\geq \min\{\, |E_\pm-E-q^2|\, \}\left(1-2q\sqrt{\frac{E_-}{(E_+-E-q^2)(E-E_-+q^2)}}\right).$$ 
\end{proposition} 

Note that this lower bound is based on information contained entirely in the eigenspectrum of original Hamiltonian (no information about the eigenvectors is needed). 
Let
$$  \overline{G}_E(\vec{x}_1,\vec{x}_2)\equiv\frac{1}{\omega_\epsilon^2}
    \int\limits_{|\vec{x} -\vec{x}_1|\leq \epsilon}d\vec{x} 
    \int\limits_{|{\bf y}-\vec{x}_2|\leq \epsilon}d{\bf y} \
    \vec{g}_E(\vec{x} ,{\bf y}),$$
where $\omega_\epsilon$ is the volume of a sphere of radius
$\epsilon$ in $\R^d$. 
We can now assemble the main result. 

\begin{theorem}
For $q$ smaller than a critical value $q_c(E)$, there exists
a constant $C_{q,E}$, independent of $\Omega$, such that:
\begin{equation}\label{UpperBound}
    |\overline{G}_E(\vec{x}_1,\vec{x}_2)|\leq C_{q,E}e^{-q|\vec{x}_1- \vec{x}_2|}.
\end{equation}
$C_{q,E}$ is given by:
\begin{equation}
    C_{q,E}=\frac{\omega_\epsilon^{-1} e^{2q\epsilon}}{\min|E_\pm-E-q^2|}\cdot  \frac{1}{1-q/F(q,E)}
\end{equation}
with
\begin{equation}
    F(q,E)=\sqrt{\frac{(E_+-E-q^2)(E-E_-+q^2)}{4E_-}}.
\end{equation}
The critical value $q_c(E)$ is the positive solution of the equation
$q=F(q,E)$.
\end{theorem}

\begin{proof}
If $\chi_{\vec{x} }$ denotes the characteristic
function of the $\epsilon$ ball centered at $\vec{x} $ (i.e.,
$\chi_{\vec{x} }(\vec{x} ^\prime)=1$ for $|\vec{x} ^\prime-\vec{x} |\leq
\epsilon$ and $0$ otherwise), then one can equivalently write
$$\overline{G}_E(\vec{x}_1,\vec{x}_2)=\omega_\epsilon^{-2}\langle
\chi_{\vec{x}_1},(H_\vec{A} -E)^{-1}\chi_{\vec{x}_2}\rangle.$$ If $\varphi_1(\vec{x} ) \equiv
e^{-\vec{q} (\vec{x} -\vec{x}_1)}\chi_{\vec{x}_1}(\vec{x} )$ and
$\varphi_2(\vec{x} ) \equiv e^{\vec{q} (\vec{x} -\vec{x}_2)}\chi_{{\bm
x}_2}(\vec{x} )$, then
\begin{eqnarray*}
    |\overline{G}_E(\vec{x}_1,\vec{x}_2)| &=&\omega_\epsilon^{-2}|\langle \varphi_1,(H_{\vec{q} ,\vec{A} }-E)^{-1}\varphi_2\rangle| e^{-\vec{q} (\vec{x}_1-\vec{x}_2)},
\end{eqnarray*}
where we used the identity: $U_\vec{q} (H_\vec{A} -E)^{-1}U_\vec{q} =(H_{\vec{q} ,\vec{A} }-E)^{-1}$. 
Choosing $\vec{q} $ parallel to $\vec{x}_1-\vec{x}_2$, we see that
$$|\overline{G}_E(\vec{x}_1,\vec{x}_2)|\leq \omega_\epsilon^{-1} e^{2q\epsilon}
e^{-q|\vec{x}_1-\vec{x}_2|}\sup\limits_{|{\bf q}|=q}\|(H_{\vec{q} ,\vec{A} }-E)^{-1}\|.$$ The statement then follows from the estimates
of Proposition \ref{Proposition:EWTT}.
\end{proof}

\subsection{Conjugate-linear symmetric operators}

Let $T:\h\to\h$ be a bounded $C$-symmetric operator and let
$A = CT$. Then $A$ is a conjugate-linear operator satisfying the symmetry condition
\begin{equation}\label{real-symmetry}
\qquad\qquad \langle A\vec{x}, \vec{y} \rangle = \langle A\vec{y}, \vec{x} \rangle, \qquad  \vec{x},\vec{y} \in \h.
\end{equation}
Indeed
\begin{equation*}
\langle A\vec{x},\vec{y}\rangle = \langle CT\vec{x},\vec{y} \rangle = \overline{ [T\vec{x},\vec{y}]} = \overline{[\vec{x}, T\vec{y}]} =
\langle C\vec{x}, T\vec{y}\rangle = \langle A\vec{y},\vec{x}\rangle.
\end{equation*}
Conversely, if $A$ is a conjugate-linear bounded operator satisfying
$$\qquad\qquad \Re \langle A\vec{x},\vec{y}\rangle = \Re \langle \vec{x}, A\vec{y} \rangle, \qquad \vec{x},\vec{y} \in \h,$$
then identity \eqref{real-symmetry} holds, simply by remarking that
$$ \Im \langle A\vec{x}, \vec{y} \rangle = \Re \langle A (i\vec{x}), \vec{y} \rangle = \Re \langle A\vec{y}, i\vec{x} \rangle =
\Im \langle A\vec{y}, \vec{x} \rangle$$ also holds true.

Thus, there is a straightforward dictionary between conjugate-linear operators that are $\R$-selfadjoint
and $C$-symmetric operators. A study of the first class, motivated by classical examples such as 
Beltrami or Hankel operators, has been vigorously pursued by the Finnish school \cite{Eirola,HuhtanenNevanlinna,HP1,HP2,HR,R1}.
We confine ourselves to reproduce below only a small portion of their results.
In particular, we discuss the adapted functional calculus for conjugate-linear operators and the related theory
of complex symmetric Jacobi matrices.

Suppose that $A$ is a bounded conjugate-linear operator and $p(z)$ is a polynomial. Then
$p(A)$ makes sense as a $\R$-linear transformation. Moreover, writing
$$p(z) = q(z^2) + z r(z^2)$$ one immediately finds that $$p(A) = q(A^2) + A  r(A^2),$$ in
which the first term is $\C$-linear and the second is conjugate-linear.
Assume that $A$ is $\R$-selfadjoint in the sense of formula \eqref{real-symmetry}. Then we
know from the refined polar decomposition (Theorem \ref{RPDB}) that $A= J |T|$, in which $T$ is a positive $\C$-linear operator and 
$J$ is a conjugation commuting with $|T|$. Thus $A^2 = |T|^2$ is a positive operator and
$$ p(A) = q(|T|^2) + J |T| r(|T|^2).$$

The spectrum $\sigma(A)$ of a conjugate-linear operator $A$ is circularly symmetric, that is,
it is invariant under
rotations centered at the origin. By passing to a uniform limit in the observation above,
one finds the following result of Huhtanen and Per\"{a}m\"{a}ki.

\begin{theorem}
Let $A$ be a bounded conjugate-linear operator which is $\R$-selfadjoint and let $A=J|T|$ be its polar
decomposition, in which $J$ is a conjugation commuting with $|T|$. For a continuous 
function $f(z) = q(|z|^2) + z r(|z|^2)$ with $q,r$ continuous on $\sigma(A^2) \subseteq [0,\infty)$
the spectral mapping theorem holds
$$ f(\sigma(A)) = \sigma(f(A)).$$
In particular,
$$ \| f(A)\| = \max_{\lambda \in \sigma(A)} |f(\lambda)|.$$
\end{theorem}

The selection of examples we present below is related to the classical
moment problem on the line, where Jacobi matrices play a central role.

\begin{example}
Let $\{\alpha_n\}$ be a bounded sequence of complex numbers and let
$\{ \beta_n\}$ be bounded sequences of positive numbers.
The associated infinite matrix
$$ \Xi = \begin{bmatrix}
\alpha_1 & \beta_1 & 0 & \ldots & 0 \\
\beta_1 & \alpha_2 & \beta_2 & \ldots & 0\\
0& \beta_2 & \alpha_3 & \ldots & 0\\
\vdots & \vdots & \ddots & & \vdots\\
0 & 0  & \ldots & \ddots & \vdots \end{bmatrix}$$
is complex symmetric.  As such, $\Xi$ is $C$-symmetric, regarded as an operator on $\ell^2(\N)$, with respect to the
standard conjugation $C (\vec{x}_n) = (\overline{\vec{x}_n}).$ Then the operator
$A = C\Xi$ is conjugate-linear and $\R$-selfadjoint in the above sense.

In view of the functional calculus carried by the operator $\Xi = J|T|$, it is natural to consider
the vector space $\mathcal{P}$ of polynomials generated by $|z|^{2n}$ and $z |z|^{2n}$. An element of
$\mathcal{P}$ is of the form
$$ f(z) = q(|z|^2) + z r(|z|^2)$$
where $q$ and $r$ are polynomials in $|z|^2$.
Let $E$
denote the spectral measure of operator $|T|$.  If
$d\mu = \langle E(d\lambda) e_1, e_1\rangle$, in which $e_1 = (1,0,0,\ldots)$, then
$$\langle q(\Xi^2) e_1, e_1 \rangle = \int_{\sigma(\Xi)} q(|\lambda|^2) d\mu(\lambda).$$
Next observe that
$$ \Re \langle f(\Xi) e_1, e_1 \rangle \geq 0 $$
whenever $\Re f|_{\sigma(\Xi)} \geq 0.$ Hence the measure $\mu$ can be extended
to a positive measure $\nu$ supported by $\sigma(\Xi)$ and satisfying
$$\qquad\qquad \langle f(\Xi) e_1, e_1 \rangle = \int_{\sigma(\Xi)} f(\lambda) d\nu(\lambda), \qquad f \in {\mathcal P}.$$
Due to the rotational symmetry of $\sigma(\Xi)$, the extension $\nu$ of $\mu$ is far from unique.

As a consequence one obtains a positive definite inner product on ${\mathcal P}$, defined by
\begin{equation}\label{eq:MPIP}
 ( f, g ) = \langle f(\Xi)e_1, g(\Xi)e_1 \rangle = \int_{\sigma(\Xi)} f \overline{g} \,d\nu.
\end{equation}
The reader will now recognize the classical relationship between orthogonal polynomials, Jacobi matrices
and positive measures. In our particular case, we obtain the recurrence relations
$$\qquad\qquad \lambda \overline{p_j(\lambda)} = \beta_{j+1} p_{j+1}(\lambda) + \alpha_j p_j(\lambda) +\beta_j p_{j-1}(\lambda), \qquad j \geq 0,$$
where $p_0, p_1, \ldots$ represent the orthonormal sequence of polynomials obtained from $1, z, |z|^2, z |z|^2, |z|^4, \ldots$ 
with respect to the inner product \eqref{eq:MPIP}. We take by convention $p_{-1} = 0$ and
$\beta_0 = 0$.
\end{example}

The framework above offers a functional model for all conjugate-linear $\R$-selfadjoint operators possesing a cyclic vector. 
Numerous details, including a numerical study of the relevant inversion formulae is contained in \cite{HP1,HP2}.

\subsection{The Friedrichs operator}\label{SectionFriedrichs}

Motivated by boundary value problems in elasticity theory, Friedrichs \cite{Friedrichs}
studied a variational problem for a compact symmetric form on
the Bergman space of a planar domain.  The bilinear from introduced  by Friedrichs is
represent against the standard $L^2$ inner product by an conjugate-linear operator
now known as the \emph{Friedrichs operator} of a planar domain. The present section,
adapted from \cite{VPSBF}, only touches one aspect of this topic, 
namely its connection to complex symmetric operators and their minimax principles.

Let $\Omega \subseteq \C $ denote a bounded, connected domain and
let $L^2_a(\Omega)$ denote the \emph{Bergman space} of $\Omega$,
the Hilbert subspace of all analytic functions in the Lebesgue
space $L^2(\Omega) = L^2(\Omega, dA)$.  The symmetric bilinear
form (see Subsection \ref{Subsection:Bilinear})
\begin{equation}\label{EquationFriedrichsForm}
  B(f,g) = \int_{\Omega}f(z)g(z)\,dA(z)
\end{equation}
on $L^2_a(\Omega) \times L^2_a(\Omega)$ was studied by Friedrichs
and others in the context of classical potential theory and planar
elasticity.  This form is clearly bounded, and it turns out that it
is compact whenever the boundary $\partial\Omega$ is $C^{1+\alpha}$
for some $\alpha >0$.
In the other direction, Friedrichs himself showed that if
$\partial \Omega$ has an interior angle of $\alpha$, then
$|\sin \alpha / \alpha|$ belongs to the essential spectrum of
the form and hence $B$ is not compact.  We assume throughout this
section that the domain $\Omega$ is chosen so that the bilinear
form $B$ is compact.

We are interested here in finding the best constant $c(\Omega) < 1$
and an optimal subspace $\V$ of $L^2_a(\Omega)$ of codimension one
for which the \emph{Friedrichs inequality}
\begin{equation}\label{EquationFriedrichsInequality}
  \left| \int_{\Omega} f^2\, dA \right| \leq c(\Omega) \int_{\Omega} |f|^2 dA
\end{equation}
holds for all $f$ in $\V$.  As we will shortly see,
the optimal constant $c(\Omega)$ is precisely $\sigma_2$,
the \emph{second} singular value of the bilinear form
\eqref{EquationFriedrichsForm}.

One important aspect of the Friedrichs inequality is that it provides
an $L^2(\Omega,dA)$ bound on harmonic conjugation.  Recall that
harmonic conjugation $u \mapsto \widetilde{u}$ (where $u$ and
$\widetilde{u}$ are real-valued harmonic functions on $\Omega$)
is well-defined only after insisting upon a certain normalization
for the conjugate functions $\widetilde{u}$.  Typically, one
requires that $\widetilde{u}$ vanishes at a certain point $z_0$ in $\Omega$.
Such requirements correspond to restricting the analytic function $f = u + i \widetilde{u}$ to lie in
a subspace $\V$ of $L^2_a(\Omega)$ of codimension one.
The fact that $c(\Omega)= \sigma_2$ in \eqref{EquationFriedrichsInequality}
yields the best possible $L^2(\Omega,dA)$ bound on harmonic conjugation:
\begin{equation*}
  \int_{\Omega} \widetilde{u}^2\,dA \leq \frac{1 + \sigma_2}{1 - \sigma_2} \int_{\Omega} u^2\,dA,
\end{equation*}
where $\widetilde{u}$ is normalized so that $u + i \widetilde{u}$ belongs
to the optimal subspace $\V$.  This follows immediately upon substituting $f = u + i \widetilde{u}$
in \eqref{EquationFriedrichsInequality} and simplifying (see the proof of Lemma \ref{LemmaSimple}
for a similar computation).

Without any further restrictions on the domain $\Omega$, the bilinear
form \eqref{EquationFriedrichsForm} is not represented by a $C$-symmetric operator
in any obvious way.  Indeed, there are few natural conjugations on the Bergman space
$L^2_a(\Omega)$ that are evident.  Although one might attempt to define a conjugation
on $L^2_a(\Omega)$ in terms of complex conjugation with respect to an orthonormal
basis of $L^2_a(\Omega)$,
such bases are notoriously difficult to describe explicitly, even for relatively
simple $\Omega$.

For any fixed conjugation $C$ on $L^2_a(\Omega)$, Lemma \ref{LemmaRepresenting}
guarantees the existence of a bounded $C$-symmetric operator $T$ representing $B$ in the sense that
$B(f,g) = [Tf,g] = \inner{f,CTg}$ for all $f,g$ in $L^2_a(\Omega)$.
In the present situation, it turns out that the
conjugate-linear operator $CT$ appearing in the preceding formula
is more natural to work with than any potential linear representing
operator $T$.

Let $P_{\Omega}:L^2(\Omega) \to L^2_a(\Omega)$ denote
the \emph{Bergman projection}, the orthogonal projection from the
full Lebesgue space $L^2(\Omega)$ onto the Bergman space
$L^2_a(\Omega)$.  The \emph{Friedrichs operator} is the conjugate-linear
operator $F_{\Omega}:L^2_a(\Omega)\to L^2_a(\Omega)$
defined by the equation
\begin{equation*}
  F_{\Omega}f = P_{\Omega}\overline{f},
\end{equation*}
which can also be written in terms of the Bergman kernel $K(z,w)$ of $\Omega$:
\begin{equation*}
  [F_{\Omega}f](z) = \int_{\Omega} K(z,w) \overline{f(w)}\,dA(w), \quad z \in \Omega.
\end{equation*}
The Friedrichs operator represents the bilinear form \eqref{EquationFriedrichsForm}
in the sense that
$$B(f,g) = \inner{f, F_{\Omega}g}$$
for all $f,g$ in $L^2_a(\Omega)$.  Indeed, this is a straightforward computation:
\begin{equation*}
  B(f,g)
  = \inner{P_{\Omega}f,\overline{g}}
  = \inner{f, P_{\Omega}\overline{g}}
  =\inner{f, F_{\Omega}g}
\end{equation*}
and hence $CT = F_{\Omega}$ for any $C$-symmetric operator $T$
representing the bilinear form $B$.  In light of the refined polar decomposition (Theorem \ref{RPDB}), we see that there exists a 
conjugation $J$ that commutes with $|T|$ and satisfies $F_{\Omega} = J|T|$.

Since $P_{\Omega}$ is a projection, it follows immediately that $0
\leq |T| \leq I$.  In fact, we can say a good deal more about
$|T|$ (or equivalently, about the symmetric bilinear form
\eqref{EquationFriedrichsForm}). We start by recalling a useful
fact, implicit in the article of Friedrichs:

\begin{lemma}\label{LemmaSimple}
  If $\Omega$ is connected, then $\sigma_1 < \sigma_0 = 1$.  In particular,
  the largest singular value of $B(x,y)$ has multiplicity one and
  the corresponding eigenfunctions are the constant functions.
\end{lemma}

\begin{proof}
  Since $F_{\Omega} = J|T|$ and $J$ commutes with $|T|$, one can find a basis
  of each spectral subspace of $|T|$ (corresponding to a non-zero
  eigenvalue) which is left invariant by $J$.
  If $f$ is such an eigenvector corresponding to the
  eigenvalue $1$, then $|T|f=f$ and $Jf=f$, which implies that
  $F_{\Omega}f = f$. Consequently
  $$\int_{\Omega} f^2\,dA = B(f,f) = \inner{f,F_{\Omega}f} = \inner{f,f} = \int_{\Omega} |f|^2\,dA.$$
  Setting $f = u + iv$ where $u$ and $v$ are real-valued and harmonic, we obtain
  \begin{align*}
    \int_{\Omega} (u^2+v^2)\,dA
    &= \int_{\Omega} (u^2 - v^2)\,dA + 2 i \int_{\Omega} uv\,dA\\
    &= \int_{\Omega} (u^2 - v^2)\,dA
  \end{align*}
  since the left hand side is real.  This implies that
  $\int_{\Omega} v^2 \, dA = 0$
  and hence $v$ vanishes identically on $\Omega$.  Since $\Omega$
  is connected and $f$ analytic, $f$ must be constant throughout $\Omega$.
  Conversely, it is clear that $\sigma_0 = 1$ since $0 \leq |T| \leq I$ and
  $F_{\Omega}$ fixes real constants.
\end{proof}

The following result demonstrates the nature of Friedrichs inequality
at the abstract level \cite{VPSBF}.

\begin{theorem}\label{TheoremAbstract}
  If $B:\h\times\h\to\h$ is a compact, symmetric,
  bilinear form with singular values
  $\sigma_0 \geq \sigma_1 \geq \cdots \geq 0$,
  repeated according to multiplicity, and corresponding unit eigenfunctions
  $\vec{e}_0, \vec{e}_1,\ldots$, then
  \begin{equation}\label{EquationFriedrichsAbstract}
    |B(\vec{x},\vec{x})| \leq \sigma_2 \norm{\vec{x}}^2
  \end{equation}
  whenever $\vec{x}$ is orthogonal to the vector
  $\sqrt{\sigma_1} \vec{e}_0 + i \sqrt{\sigma_0} \vec{e}_1$.
  Furthermore, the constant $\sigma_2$ in \eqref{EquationFriedrichsAbstract}
  is the best possible for $\vec{x}$ restricted to a subspace of $\h$ of codimension one.
\end{theorem}

In essence, \eqref{EquationFriedrichsAbstract} provides the best
possible bound on a symmetric bilinear form that can be obtained
on a hyperplane which passes through the origin.
Since the orthogonal complement of the vector
$(\sqrt{\sigma_1} \vec{e}_0 - i \sqrt{\sigma_0} \vec{e}_1 )$also  has the same
property, we see that the optimal subspace in Theorem \ref{TheoremAbstract} is not unique.

\subsection{Asymptotics of eigenvalues of compact symmetric bilinear forms} The example of the Friedrichs operator
discussed in the previous section is only one instance of a more general framework. We reproduce below from
\cite{PutinarProkhorov} a few abstract notions and facts, with the direct aim at illuminating some aspects of the asymptotic analysis
of the spectra of compact symmetric bilinear forms.

Let $\h$ be a complex separable Hilbert space and let $B(\vec{x},\vec{y})$ be a
compact bilinear symmetric form on $\h$.  Following the discussion in Subsection \ref{Subsection:Bilinear},
the singular values of $B$ (also called the \emph{characteristic values} of $B$)
form a decreasing sequence $\lambda_0 \geq \lambda_1 \geq \ldots \geq 0$
and we can find a sequence of associated vectors $\{ \vec{u}_n\}$ that are 
characterized by the double orthogonality conditions:
\begin{equation}\label{07.28.06.1}
B(\vec{u}_n, \vec{u}_m ) = \lambda_n \delta_{mn}, \qquad \langle \vec{u}_n,\vec{u}_m\rangle = \delta_{mn}.
\end{equation}
These vectors are obtained as eigenvectors, fixed by the auxiliary conjugation $J$ (Theorem \ref{RPDB}) 
of the modulus $|T|$ of any $C$-symmetric representing operator
$T$ satisfying $B(\vec{x},\vec{y}) = \inner{ T\vec{x}, C\vec{y}}$, as in Lemma \ref{LemmaRepresenting}.

The following variant of Weyl-Horn estimate is the root of all asymptotic evaluations of the
distribution of the characteristic values of $B$.

\begin{proposition}
\label{th1} Let $B(\,\cdot\,,\,\cdot\,)$ be a compact bilinear symmetric form on a
complex Hilbert space $\h$ and let $\lambda_0 \geq \lambda_1 \geq \ldots \geq 0$ denote its sequence of
characteristic values. Let $\vec{g}_0, \vec{g}_1,\ldots,\vec{g}_n$ be a system of vectors
in $\h$. Then for any nonnegative integer $n$,
\begin{equation}
| \det (B(\vec{g}_i,\vec{g}_j)) | \leq \lambda_0 \lambda_1 \cdots\lambda_n \det (\langle
\vec{g}_i, \vec{g}_j \rangle).
\label{06.15.06.1}
\end{equation}
\end{proposition}

\begin{proof}
Let $\{ \vec{u}_k\}$ denote an orthonormal system satisfying \eqref{07.28.06.1}.
Write
$$\qquad\qquad \vec{g}_i = \sum_{k=0}^\infty c_{ik} \vec{u}_k, \qquad 0 \leq i \leq n.$$
Then
$$ B(\vec{g}_i, \vec{g}_j) = \sum_k c_{ik} c_{jk} B(\vec{u}_k, \vec{u}_k) = \sum_k \lambda_kc_{ik}
c_{jk}.$$
Therefore
\begin{align*}
| \det (B(\vec{g}_i,\vec{g}_j)) |
&= \frac{1}{(n+1)!} \sum_{k_0,\dots,k_n}
\lambda_{k_0} \cdots\lambda_{k_n} (\det(c_{ik_j}))^2 \\
&\leq
 \lambda_0 \lambda_1 \cdots\lambda_n \frac{1}{(n+1)!} \sum_{k_0,\dots,k_n}
 |\det(c_{ik_j})|^2 \\
 &= 
  \lambda_0 \lambda_1 \cdots\lambda_n \det \Big(\sum_k c_{ik}
 \overline{c_{jk}}\Big)\\
 & =  \lambda_0 \dots\lambda_n \det (\langle
\vec{g}_i, \vec{g}_j \rangle).\qedhere
\end{align*}
\end{proof}

A cousin of the preceding result is stated below, as a compact bilinear symmetric form variant of the Ky Fan
inequality.
\begin{proposition}
Let $B(\,\cdot\,, \,\cdot\,)$ be a compact bilinear symmetric form on a complex Hilbert space $\h$ and let 
$\lambda_0 \geq \lambda_1 \geq \ldots \geq 0$ denote its sequence of characteristic values. 
Then for any orthonormal system $\vec{g}_0, \vec{g}_1,\dots, \vec{g}_n$ of vectors in $\h$
$$\left|\sum_{i=0}^n B(\vec{g}_i, \vec{g}_i)\right| \le \lambda_0 + \lambda_1 + \dots +\lambda_n.$$
\end{proposition}

\begin{proof}
Let $\{ \vec{u}_k\}$ denote an orthonormal system satisfying \eqref{07.28.06.1}.
We have
$$\vec{g}_j=\sum_{k=0}^\infty c_{ik} \vec{u}_k,\ \ 0 \leq j \leq n, $$
\begin{equation}
\left\langle \vec{g}_i, \vec{g}_j \right\rangle=\sum_k c_{ik}
\overline{c_{jk}} =\delta_{ij}, \label{06.15.06.2}
\end{equation}
and
$$B(\vec{g}_i, \vec{g}_j)=\sum_k \lambda_k c_{ik} c_{jk}.$$
It is easy to see that
$$\left|\sum_{i=0}^n B(\vec{g}_i, \vec{g}_i)\right|=\left|\sum_{i=0}^n \sum_k \lambda_k c^2_{ik}\right| \le \sum_{i=0}^n \sum_k \lambda_k |c_{ik}|^2.$$
Let us consider the following polynomial of degree $n+1$:
\begin{equation}
P(\lambda) =\det \Big( \sum_k (\lambda_k -\lambda) c_{ik}\overline{c_{jk}} \Big).
\label{06.15.06.3}
\end{equation}
As above, one finds
$$P(\lambda) =\frac{1}{(n+1)!} \sum_{k_0, \dots, k_n} (\lambda_{k_0} -\lambda)\dots (\lambda_{k_n}-\lambda) |\det (c_{ik_j})|^2.$$
From this, by (\ref{06.15.06.2}) and (\ref{06.15.06.3}), we infer
\begin{align*}
\sum_{i=0}^n \sum_k \lambda_k |c_{ik}|^2 
&= \frac{1}{(n+1)!} \sum_{k_0, \dots, k_n} (\lambda_{k_0}+\dots +\lambda_{k_n})|\det(c_{ik_j})|^2\\
&\leq (\lambda_0+\dots +\lambda_n) \frac{1}{(n+1)!} \sum_{k_0, \dots, k_n} |\det(c_{ik_j})|^2\\
&= (\lambda_0+\dots +\lambda)\det (\left\langle \vec{g}_i, \vec{g}_j\right\rangle)\\
&=\lambda_0+\dots+\lambda_n.\qedhere
\end{align*}

\end{proof}

Examples abound. We relate the above to the Friedrichs operator studied in Subsection \ref{SectionFriedrichs}, as follows.
Let $\Omega$ be a bounded open subset of the complex plane, with the analytic quadrature identity
$$ \int_\Omega f(z) dA(z) = \int_K f(z) d\mu(z),$$
where $f$ is an analytic function defined on the closure of $\Omega$ and $\mu$ is a positive measure supported by
a compact set $K \subseteq \Omega$. We will work with Friedrichs' bilinear form defined on Bergman space:
$$ \qquad\qquad B(f,g) = \int_\Omega f g \,dA, \qquad f, g \in L^2_a(\Omega).$$
The compactness of the form $B$ follows from Montel's Theorem.
Putting together the preceding inequalities one obtains the asymptotic behavior of the eigenvalues
$\lambda_n$ of Friedrichs' form.

\begin{theorem}
Let $\Omega$ be a planar domain carrying an analytic quadrature identity given by
a positive measure $\mu$ supported by the compact set $K \subseteq \Omega$ .  Then:

$$
\limsup_{n\to \infty} (\lambda_0 \lambda_1 \dots \lambda_n)^{1/n^2} \le \exp(-1/C(\partial \Omega, K)),
$$
where $C(\partial \Omega, K)$ is the capacity of the condenser $(\partial \Omega, K)$,
$$ \limsup_{n \to \infty} \lambda_n^{1/n} \le \exp(-1/C(\partial \Omega, K)),
$$
and
$$
\liminf_{ n\to \infty} \lambda_n^{1/n} \le \exp (-2/C(\partial \Omega, K)).
$$
\end{theorem}

The proof, and other similar examples of asymptoitics of the eigevalues of compact bilinear symmetric forms
are contained in \cite{PutinarProkhorov}.

\subsection{The Neumann-Poincar\'e operator in two dimensions}

The classical boundary problems for harmonic functions
can be reduced to singular integral equations on the boundary of the respective domain
via single and double layer potentials. The double layer potential, also known as the
Neumann-Poincar\'e operator, offers an elegant path for solving such boundary problems and
at the same time it is one of the most important and well studied singular integral operators
\cite{Ahl,Plemelj}. The spectrum of the Neumann-Poincar\'e operator coincides, up to normalization, with the Fredholm 
eigevalues of the underlying domain, providing important invariants in quasi-conformal mapping theory.
Two real dimensions are special, due to the existence of complex variables and
the harmonic conjugate of a harmonic functions. An intimate relationship
between the Neumann-Poincar\'e operator and a $C$-symmetric operator, acting on the 
underlying Bergman space, was discovered by Schiffer \cite{Sch1,Sch2}. We illustrate,
from the restricted point of view of our survey,  this connection. Complete details can be found in
\cite{KPS}.

Let  $\Gamma$ be
$C^2$-smooth Jordan curve, surrounding the domain $\Omega \subseteq
\C$, and having $\Omega_e$ as exterior domain. We denote
by $z,w,\zeta,\ldots$ the complex coordinate in $\C$ and by
$\partial_{\overline{z}} = \frac{\partial}{\partial \overline{z}}$
the Cauchy-Riemann operator. The area measure will be
denoted $dA$. Following Poincar\'e, we consider the space $\mathfrak{H}$ consists of (real-valued)
harmonic functions $h$ on $\C \setminus \Gamma$ having
square summable gradients:
$$ h \in \mathfrak{H} \Leftrightarrow \int_{\Omega \cup \Omega_e} |\ \partial_{\overline z} h(z)|^2 dA(z) <\infty, \ h(\infty) =0.$$
Note that the gradients $\partial_{\overline z} h $ are now square
summable complex conjugate-analytic functions. The gradients of elements in $\mathfrak{H}_i$ form the Hilbert space
$\mathfrak{B}(\Omega)$, which is the complex conjugate of the
Bergman space $L_a^2(\Omega)$ of $\Omega$. Boundary values will be considered in appropriate fractional order Sobolev spaces
$W^s(\Gamma)$.

 The Hilbert space $\mathfrak{H}$ possesses two natural direct sum decompositions:
$$ \mathfrak{H} = \mathfrak{S} \oplus \mathfrak{D} = \mathfrak{H}_i \oplus \mathfrak{H}_e.$$
The first one corresponds to the ranges of the single $S_f$, respectively double $D_f$, layer potentials of 
charge distributions $f$ on the boundary $\Gamma$. The second subspaces are
$$ \mathfrak{H}_i = \{ (h_i,0) \in \mathfrak{H}\}, \ \ \ \mathfrak{H}_e = \{ (0,h_e) \in \mathfrak{H}\}.$$

The single and double layer potentials are in this case strongly
related to Cauchy's integral. For instance, the singular integral component of the double layer potential is
$$ (Kf)(z) = \int_\Gamma f(\zeta) \Re  \Big[\frac{d \zeta}{2 \pi i (\zeta - z)} \Big] = \frac{1}{2 \pi}
\int_\Gamma f(\zeta)\  d \ {\rm arg}(\zeta - z).$$

The following complex conjugate-linear singular integral operator plays
the role of the symmetry $P_d-P_s$ in our notation. Let $F =
\nabla S_f$, for $f \in W^{1/2}(\Gamma)$, be regarded as a single
conjugate-analytic function defined on all $\Omega \cup \Omega_e$.
Define the Hilbert (sometimes called Beurling) transform
\begin{equation*}
(T F )(z) = {\rm p.v.} \frac{1}{\pi } \int_{\Omega \cup \Omega_e} \frac{\overline{F(\zeta)}}
{(\overline{\zeta}-\overline{z})^2} dA(\zeta).
\end{equation*}

\begin{lemma} Let $h \in \mathfrak{H}$ be represented as
$h = D_f + S_g$, in which $f \in W^{1/2}(\Gamma)$ and $g \in W^{-1/2}(\Gamma)$. Then
\begin{equation*} 
T \nabla (D_f+ S_g) = \nabla ( D_f - S_g).
\end{equation*}
\end{lemma}

\begin{corollary} The conjugate-linear transform $T$ is an isometric isomorphism of
the space $\mathfrak{B}(\Omega) \oplus \mathfrak{B}(\Omega_e)$ onto itself.
\end{corollary}

We are ready to define the principal conjugate-linear operator for our study:
$$T_\Omega : \mathfrak{B}(\Omega) \to \mathfrak{B}(\Omega), \ \ T_\Omega (F) (z) = T (F,0) (z),
\ \ z \in \Omega,$$ where $(F,0)$ means the extension of $F \in
\overline{L_a^2(\Omega)}$ by zero on $\Omega_e$. Thus the operator
$T_\Omega$ and the one described above coincide as linear
transformations over the real field. 

A key observation, going back to the pioneeering work of Poincar\'e, is that 
the angle operator
$ P_s(P_e-P_i)P_s$ measuring the balance of energies (inner-outer) of a harmonic field
generated by a single layer potential is unitarily equaivelent to $K$, see for details
\cite{KPS}. But it is a simple matter of the geometry of Hilbert spaces that the 
angle operator $P_i(P_d-P_s)P_i$ is unitarily equaivalent to $ P_s(P_e-P_i)P_s$.
We are led to the following nontrivial consequences, originally proved by Schiffer
\cite{Sch1, Sch2}.

\begin{theorem} 
Let $\Omega$ be a bounded planar domain with $C^2$ smooth boundary
and let $T_\Omega : L_a^2(\Omega) \to L_a^2(\Omega)$ be the conjugate-linear operator
$$ [T_\Omega f](z) = {\rm p.v.} \frac{1}{\pi } \int_{\Omega} \frac{\overline{f(\zeta)}}
{({\zeta}-{z})^2} dA(\zeta), \ \ f \in A^2(\Omega), \ z \in \Omega.$$
Then $T_\Omega$ is compact and the eigenvalues of the conjugate-linear eigenvalue problem
$$ T_\Omega f_k = \lambda_k f_k$$
coincide (multiplicities included) with the spectrum of the Neumann-Poincar\'e
operator $K$, except the eigenvalue $1$. The eigenfunctions $\{f_k\}$ are orthogonal and complete
in $L_a^2(\Omega)$.
\end{theorem}

In particular one finds that
\begin{equation}
\| T_\Omega \| = \lambda_1^+,
\end{equation}
where $\lambda_1^+$ is the largest eigenvalue of $K$ less than $1$.

Note the ambiguity of phase in the eigenvalue problem $T_\Omega f
= \lambda f$. By multiplying $f$ by a complex number $\tau$ of
modulus one, the complex conjugate-linearity of $T_\Omega$ implies
$T_\Omega f = \tau^2 \lambda f.$ On the other hand, we have
identified $T$ with an $\R$-linear operator ($P_d-P_s$)
acting on gradients of real harmonic functions. This simple
observation leads to the following characteristic symmetry of the
Neumann-Poincar\'e operator specific for two variables.

\begin{proposition} Let $\Gamma \subseteq \R^2$ be a $C^2$-smooth Jordan curve. Then, except the point $1$,
 the spectrum of the
Neumann-Poincar\'e operator acting on $L^2(\Gamma)$ is symmetric with respect to the origin, multiplicities included:
$\lambda \in \sigma(K), \lambda <1$ if and only if $-\lambda \in \sigma(K)$.
\end{proposition}

\begin{proof} Let $\lambda \in \sigma({K}) \setminus \{ 1\}$ and let $(u,0) \in \mathfrak{H}$
be the associated eigenfunction of the operator $P_i(P_d-P_e)P_i$. By the above correspondence there exists an
anti-analytic function $F = \partial_{\overline z} u$ satisfying
$T_\Omega F = \lambda F$. Let $G= iF$ and remark that the
conjugate-linearity of $T_\Omega$ implies $T_\Omega G = -\lambda G$.
Remark also that $G = \partial_{\overline z} \widetilde{u},$ where
$\widetilde{u}$ is the harmonic conjugate of $u$. Thus, the
eigenvector in $\mathfrak{H}$ corresponding to the eigenvalue
$-\lambda$ is simply $(\widetilde{u},0)$.
\end{proof}

Another symmetry is also available from the above framework.

\begin{proposition} Let $\Omega$ be a bounded planar domain with $C^2$-smooth boundary
and let $\Omega_e$ be the exterior domain. Then the Bergman space operators $T_\Omega$
and $T_{\Omega_e}$ have equal spectra.
\end{proposition}

\begin{proof} Let $(F,0)$ be an eigenvector of $T_\Omega$, corresponding to the eigenvalue $\lambda$.
Denote $T (F,0) = (\lambda F, G)$. Since $T^2 = I$ we get $(F,0) = \lambda T(F,0) + T(0,G) = (\lambda^2 F, \lambda G) + T(0,G)$.
Thus $T(0,G) = ((1-\lambda^2)F, -\lambda G)$. This means $-\lambda \in \sigma(T_{\Omega_e})$ and by the preceding symmetry principle
$\lambda \in  \sigma(T_{\Omega_e})$.
\end{proof}

We can assert with confidence that most of Schiffer (and collaborators) works devoted to the Fredholm spectrum of a planar domain are, although not stated as such,
consequences of the obvious unitary equaivalence between the angle operators $ P_s(P_e-P_i)P_s$ and $ P_s(P_e-P_i)P_s$ \cite{Sch1,Sch2}.

{
\subsection{Symmetrizable operators} A great deal of effort was put in the physics community for deriving from the
$C$-symmetry of an operator $T$,
$$ T^\ast C = CT$$
the reality of its spectrum. Almost all studies starting by a rescaling of the Hilbert space metric 
with the aid of a positive operator of the form
$$ A = CS >0,$$
where $S$ is bounded and commutes with $T$. Indeed, in this case
\begin{equation}\label{similarity}
T^\ast A =  T^\ast CS = C T S = CS T = AT,
\end{equation}
or in equivalent terms
$$ \langle A Tf,g\rangle = \langle A f, Tg \rangle.$$
Non-selfadjoint operators with this property are called \emph{symmetrizable}.
In general, but not always, the operator $A$ is assumed to be invertible. In case $A$ is only one-to one, non-negative
it has a dense range in the underlying Hilbert space, so that the sequilinear form $\langle Af,g\rangle$ defines
a norm which is not equivalent to the original one. The latter framework is the origin of the concept of generalized function in a Gelfand triple of Hilbert spaces, with its known impact in diagonalizing concrete unbounded operators.
Far from being exhaustive, we refer to the following list of works  relating $\mathcal{PT}$-symmetric operators to symmetrizable ones \cite{SGH1992,Znojil1,Zn2004,Znojil2012,AK1, AK2, AK3, AGK}.

{ Symmetrizable operators appear in many physics contexts. As explained in the work by Scholtz and collaborators \cite{SGH1992}, even in the traditional formulation of Quantum Mechanics, there are important situations when one has to deal with non-selfadjoint operators. This is the case, for example, for the effective quantum models obtained by tracing out a number of degrees of freedom of a large quantum system, an operation leading to non-selfadjoint physical observables. Such effective models can be soundly  interpreted and analyzed if the physical observables are symmetrizable.  The authors of \cite{SGH1992} went on to formulate the following problem: Given a set of non-selfajdoint observables $T_i$ that are simultaneously symmetrizable by the same $A$, i.e., $T_i^\ast A = A T_i$ for all $T_i$'s,  in which conditions is the ``metric operator" $A$ uniquely defined? The issue is important because the expected values of the observables are physically measurable and they must be un-ambiguously defined. The uniqueness of $A$ will ensure that through the rescaled Hilbert space metric by $A$.  The answer to this question, which is quite satisfactory from a physical point of view, is as follow: The metric operator $A$ is uniquely defined by the system of $T_i$'s if and only if the set of these observables is irreducible, that is, if the only operator (up to a scaling factor) commuting with all $T_i$'s is the identity operator.

} 

The rescaling of norm idea is however much older, with roots in potential theory.
As a continuation of the preceding section we briefly recount here this classical framework which has inspired several generations of mathematicians but apparently did not reach the $\mathcal{PT}$-community.

Let $\Omega$ be a bounded domain in $\R^d$ with boundary $\Gamma$.
We assume that $\Gamma$ is at least $C^2$-smooth. The $(d-1)$-dimensional surface measure on $\Gamma$
is denoted by $d\sigma$ and the unit outer normal to a point $y \in \Gamma$
will be denoted $n_y$. We denote by $E(x,y)=E(x-y)$ the normalized
Newtonian kernel:
$$
E(x,y) = \begin{cases}
           \frac{1}{2\pi} \log \frac{1}{|x-y|}, & d=2,\\
           c_d |x-y|^{2-d}, & d \geq 3,
         \end{cases}
$$
where $c_d^{-1}$ is the surface area of the unit sphere in
$\R^d$. The signs were chosen so that $\Delta E = -\delta$
(Dirac's delta-function).

For a $C^2$-smooth function (density) $f(x)$ on $\Gamma$ we form the fundamental potentials:
the {\it single and double layer potentials} in $\R^d$; denoted $S_f$ and $D_f$ respectively:
$$
\begin{aligned}
S_f(x) = \int_\Gamma E(x,y) f(y) d\sigma(y)\\
D_f(x) = \int_\Gamma \frac{\partial}{\partial n_y} E(x,y) f(y) d \sigma(y).
\end{aligned}
$$

The Neumann-Poincar\'e kernel, appearing in dimenion two in the preceding section,
$$ K(x,y):=  -\frac{\partial}{\partial n_y} E(x-y); \ \  K^\ast(x,y) =  -\frac{\partial}{\partial n_x} E(x-y)$$
satisfies growth conditions which insure the compactness of the associated integral operator acting on the boundary:
$$ (Kf)(x) = 2 \int_\Gamma K(x,y) f(y) d\sigma(y), \ \ f \in L^2(\Gamma, d\sigma). $$
Similarly, the linear operator $$ Sf = S_f|_\Gamma, \ \ f \in
L^2(\Gamma),$$ turns out to be bounded (from $L^2(\Gamma)$ to the
same space). Remark that the representing kernel $E(x,y)$ of $S$
is pointwise  non-negative for $d\geq 3$. As a matter of fact the total energy of the field generated by 
the pair of harmonic functions $S_f$ (in $\Omega$ and its complement) is $\langle Sf, f\rangle_{2,\Gamma}$.

Returning to the main theme of this section, the following landmark observation, known as {\it Plemelj'
symmetrization principle} unveils the reality of the spectrum of the Neumann-Poincar\'e operator $K$, see
\cite{Plemelj}.
For a modern proof and details we refer to \cite{KPS}.

\begin{theorem} The layer operators $S,K : L^2(\Gamma) \longrightarrow L^2(\Gamma)$ satisfy the identity
\begin{equation}\label{layer}
KS = SK^\ast .
\end{equation}
\end{theorem}

For an early discussion of the  importance of the above rescaling identity in potential theory see \cite{Korn1917,Mercer1919}.
It was however Carleman who put Plemelj' symmetrization principle at work, in his remarkable dissertation
focused on domains with corners \cite{Carleman}.

Numerous authors freed the symmetrization principle from its classical field theory roots, to mention only
\cite{KreinCompact,Lax,Zaanen}. We reproduce only Kre\u{\i}n's observation, which potentially can
impact the spectral analysys of unbounded $C$-symmetric operators via their resolvent.

Let $H$ be an infinite dimensional, separable, complex Hilbert
space and let ${\mathcal C}_p = {\mathcal C}_p (H), p \geq 1, $ be
the Schatten-von Neumann class of compact operators acting on $H$.

\begin{theorem} Let $p \geq 1$ and let $M \in {\mathcal C}_p (H)$ be a linear bounded
operator with the property that there exists a strictly positive bounded
operator $A$ such that $R M = M^\ast R$.

Then the spectrum of $M$ is real and
for every non-zero eigenvalue $\lambda$, if $(M-\lambda)^m f =0$ for some $m>1$,
then $(M-\lambda)f = 0$.

Moreover, the eigenvectors of $M^\ast$, including the null vectors, span $H$.
\end{theorem}

The above theorm directly applies to the Neumann-Poincar\'e operator $K$ and, in the case of dimension two,
to the Beurling transform $T_\Omega$ discussed in the preceding section, cf.  \cite{KPS}.

}

\bibliographystyle{plain}
\bibliography{JPA-ReviewArticle}

\end{document}